\documentclass[11pt]{amsart}

\usepackage{amsxtra,amssymb,amsmath,amscd,url,listings,stmaryrd}
\usepackage[alphabetic]{amsrefs}
\usepackage[utf8]{inputenc}
\usepackage{eucal}
\usepackage{fullpage}
\usepackage{scrtime}
\usepackage[colorlinks]{hyperref}
\usepackage{datetime}
\usepackage[all]{xy}
\usepackage{graphicx}
\usepackage{tikz-cd}

\shortdate
\usepackage{setspace}

\numberwithin{equation}{section}

\newcommand{\MFq}{M_2(\mathbb F_q)}

\theoremstyle{plain}
\newtheorem{theorem}{Theorem}[section]
\newtheorem*{theorem*}{Theorem}
\newtheorem{lemma}[theorem]{Lemma}
\newtheorem{corollary}[theorem]{Corollary}

\newtheorem{problem}[theorem]{Problem}
\newtheorem{proposition}[theorem]{Proposition}

\newtheorem{notation}[theorem]{Notation}
\theoremstyle{remark}

\theoremstyle{definition}

\newtheorem{definition}[theorem]{Definition}

\begin{document}

\title{Distribution of determinant of sum of matrices}

\author{Daewoong Cheong}
\address{Department of Mathematics\\
Chungbuk National University \\
Cheongju, Chungbuk 28644 Korea}
\email{daewoongc@chungbuk.ac.kr}

\author{Doowon Koh}
\address{Department of Mathematics\\
Chungbuk National University \\
Cheongju, Chungbuk 28644 Korea}
\email{koh131@chungbuk.ac.kr}

\author{Thang Pham} \address{Department of Mathematics\\
University of Rochester\\
Rochester, NY 14627 USA}
\email{vanthangpham@rochester.edu}

\author{Le Anh Vinh}
\address{Vietnam Institute of Educational Sciences\\ Hanoi, 100000, Vietnam}
\email{vinhle@vnies.edu.vn}

\begin{abstract}
Let $\mathbb{F}_q$ be an arbitrary finite 
field of order $q$. In this article, we study $\det S$ for certain types of subsets $S$ in the ring  $M_2(\mathbb F_q)$ of $2\times 2$ matrices with entries in $\mathbb F_q$. 
For $i\in \mathbb{F}_q$, let $D_i$ be the subset of $M_2(\mathbb F_q)$ defined by
 $ D_i := \{x\in M_2(\mathbb F_q): \det(x)=i\}.$ Then our results can be stated as follows.
First of all, we show  that when $E$ and $F$ are subsets of $D_i$ and $D_j$ for  some $i, j\in \mathbb{F}_q^*$, respectively,  we have $$\det(E+F)=\mathbb F_q,$$
whenever  $|E||F|\ge {15}^2q^4$, and then  provide a concrete construction to show that our result is sharp.  Next, as an application of the first result,  we  investigate a distribution of the determinants generated by the sum set $(E\cap D_i) + (F\cap D_j),$ when $E, F$ are subsets  of the product type, i.e., $U_1\times U_2\subseteq \mathbb F_q^2\times \mathbb F_q^2$ under the identification $ M_2(\mathbb F_q)=\mathbb F_q^2\times \mathbb F_q^2$.  Lastly, as an extended version of the first result,  
we prove that if $E$ is a set in $D_i$ for $i\ne 0$ and $k$ is large enough, then we have 
\[\det(2kE):=\det(\underbrace{E + \dots + E}_{2k~\textup{terms}})\supseteq \mathbb{F}_q^*,\]
whenever the size of $E$ is close to $q^{\frac{3}{2}}$. Moreover, we show that, in general, the threshold $q^{\frac{3}{2}}$ is best possible. Our main method is based on the discrete Fourier analysis.
\end{abstract}



\thanks{The first and second listed authors were supported by Basic Science Research Programs through National Research Foundation of Korea (NRF) funded by the Ministry of Education (NRF-2018R1D1A3B07045594 and NRF-2018R1D1A1B07044469, respectively). The third listed author was supported by Swiss National Science Foundation grant P400P2-183916.}

\maketitle 

\section{Introduction}
Let $E$ be a finite subset of $\mathbb R^d, d\ge 2.$ The  Erd\H{o}s distinct distances problem is to find the best possible lower bound of the distance set $\Delta(E)$ in terms of $|E|,$ where  
$\Delta(E)$ is defined as
$$ \Delta(E):=\{|x-y|: x,y \in E\}.$$
In dimension two,  Erd\H{o}s \cite{Er46} conjectured that $|\Delta(E)|\gg |E|/\sqrt{\log|E|}$. This was solved up to logarithmic factor by Guth and Katz \cite{GK15}. Indeed, they proved that $|\Delta(E)|\gg |E|/\log|E|.$ In higher dimensions, it was also conjectured by  Erd\H{o}s \cite{Er46} that $|\Delta(E)|\gg |E|^{2/d},$ which has long stayed  unsettled.   We refer readers to \cite{SV04, SV08} for recent developments and 
partial results on the Erd\H{o}s distinct distances problem in three and higher dimensions. 
As a continuous analog of the Erd\H{o}s distinct distances conjecture,  Falconer \cite{Fa86} conjectured that  any subset $E$ of $\mathbb R^d$ of  the Hausdorff dimension  greater than $d/2$ determines  
a distance set of a positive Lebesgue measure. This conjecture is still open in all dimensions, and, recently, much progress on this problem has been made (see, for example, \cite{Ma87, Bo94, W099, Erd05, Erd06, DGOWWZ, GIOW}).\\ 
   
In the finite field setting, the distance problems turn out to have features of both the Erd\H{o}s and Falconer distance problems.
Bourgain-Tao-Katz \cite{BKT04}  studied the finite field Erd\H{o}s distance problem for the first time. Let $\mathbb F_q^d, d\ge 2,$ be the $d$-dimensional vector space over a finite field $\mathbb F_q$ with $q$ elements. Throughout this paper, we assume that $q$ is an odd prime power. Given two subsets $E, F$ of $\mathbb F_q^d,$ the distance set, denoted by $\Delta(E, F)$, is defined as
$$ \Delta(E, F):=\{\|x-y\|: x\in E, y\in F\},$$
where $\|\alpha\|=\alpha_1^2+ \cdots +\alpha_d^2$ for $\alpha=(\alpha_1, \ldots, \alpha_d).$ The first non-trivial result was obtained by  Bourgain-Tao-Katz \cite{BKT04}  using arithmetic-combinatorial methods and the connection of the geometric incidence problem of counting distances with sum-product estimates. They showed that if $q\equiv 3 \mod{4}$ is a prime and $E$ is a subset of $\mathbb F_q^2$ with $|E|= q^{2-\delta}$ for some $0<\delta <2,$ then there exists a positive number $\epsilon=\epsilon{(\delta)}$ such that 
$$ |\Delta(E,E)|\ge |E|^{\frac{1}{2}+ \varepsilon}.$$
In their proof of this result, it was not trivial to find an explicit relationship between $\delta$ and $\epsilon.$ Furthermore,  
as pointed out in \cite{IR07}, their result could not be extended over an arbitrary finite field. Indeed, if $q=p^2$ for a prime $p,$ then taking $E=\mathbb F_p\times \mathbb F_p,$  we have $|\Delta(E,E)|=p=q^{1/2}$ and $|E|=p^2=q.$ Moreover, if $q\equiv 1 \mod{4},$ then there exists $i\in \mathbb F_q$ with $i^2=-1$ so that the set $E=\{(t, it): t\in \mathbb F_q\}$ satisfies that $|E|=q$ and $\Delta(E,E)=\{0\}.$\\

Over an arbitrary finite  field, not necessarily a prime field, it was
 Iosevich and Rudnev \cite{IR07} who obtained an explicit lower bound  on the size 
of $\Delta(E,E)$ in terms of the size of $E$. 
More precisely, they proved that if $E\subseteq \mathbb F_q^d$ 
such that $|E|\ge Cq^{d/2}$ for a sufficiently large constant $C,$ then 
\begin{equation}\label{IREr} |\Delta(E,E)|\gg \min\left\{q,~ \frac{|E|}{q^{(d-1)/2}}\right\}.\end{equation}
Here, and throughout this paper, $X\gg Y$ means that there is a constant $C$ independent of $q$ such that $CX \ge Y$ and we also write $Y\ll X$ for $X\gg Y.$ In addition,  $X\sim Y$ is used to indicate that $X\gg Y$ and $Y\gg X.$  Shparlinski \cite{Sh06} extended the result \eqref{IREr} 
to the case when $E, F$ are arbitrary subsets of $\mathbb F_q^d$:
$$ |\Delta(E,F)| > \frac{1}{2} \min\left\{ q,~ \frac{|E||F|}{q^d}\right\}.$$
Similar results were obtained for  generalized distances defined by certain polynomials (see, for example, \cite{IK08, KS12, Vi13}).
In specific ranges of sizes of sets $E,F$ in $\mathbb F_q^d,$ slightly better lower bounds were given in \cite{Di13, KS15}.\\

Notice that the above Shparlinski's result implies that if $E,F\subseteq \mathbb F_q^d$ with $|E||F|\ge q^{d+1},$ then the distance set $\Delta(E,F)$ contains a positive proportion of all possible distances. This can be considered as a result on a finite field version of the Falconer distance problem. 

In view of these examples, Iosevich and Rudnev posed the following problems.
\begin{problem} [The Erd\H{o}s-Falconer distance problem] Let $E, F$ be subsets of $\mathbb F_q^d.$ How much large sets $E, F$ do we need to assure that  the distance set $\Delta(E,F)$ contains a positive proportion of all distances?
\end{problem}
Iosevich and Rudnev \cite{IR07} also raised the following question which calls for much stronger conclusion than in the Erd\H{o}s-Falconer distance problem.
\begin{problem}[The Strong Erd\H{o}s-Falconer distance problem] Let $E,F$ be subsets of $\mathbb F_q^d.$ What is the smallest exponent $\alpha$ such that if $|E||F|\ge C q^\alpha$, then the distance set $\Delta(E,F)$ contains all distances? \end{problem}

When $E=F,$  Iosevich and Rudnev \cite{IR07} proved that if $E\subseteq \mathbb F_q^d, d\ge 2,$ and $|E|\ge 4 q^{(d+1)/2}$, then $\Delta(E,E)=\mathbb F_q.$ The authors in \cite{HIKR10} constructed an example to show that the exponent $(d+1)/2$ in odd dimensions can not be improved without further restrictions. In even dimensions, it is conjectured that  any subset $E$ of $\mathbb F_q^d$ with $|E|\ge C q^{d/2}$ determines all distances. This conjecture is open in all even dimensions and the exponent $(d+1)/2,$ due to Iosevich and Rudnev, has not been improved in all even dimensions. There have been recently produced much related results for which we refer  to \cite{CEHIK12, KSE}.
\bigskip

On the other hand, after Iosevich and Rudnev's work,  the Erd\H{o}s-Falconer type distance problem has been studied for other geometric objects (see, for instance, \cite{hart1,vinh, han}).
Among other things,  a similar question has been addressed in the setting of matrix rings. 
For an integer $n\ge 2$, let $M_n(\mathbb{F}_q)$ be the set of $n\times n$ matrices with entries in $\mathbb{F}_q$ and $SL_n(\mathbb{F}_q)$ be the special linear group in $M_n(\mathbb{F}_q)$. 
Ferguson, Hoffman, Luca, Ostafe, and Shparlinski \cite{shpa} studied the following problem. 

\begin{problem}\label{qs1} Let $E$ and $F$ be sets in $M_n(\mathbb{F}_q).$ How large do $E$ and $F$ need to be to guarantee that there exists $(x, y)\in E\times F$ such that $\det(x+y)=1$? 
\end{problem}
Ferguson et al. \cite{shpa} developed a version of the Kloosterman sum over matrix rings to  prove that if $|E||F|\ge 2 q^{2n^2-2}$, then there exist elements $x\in E$ and $y\in F$ such that $\det(x+y)=1$. In the paper \cite{li}, Li and Hu gave an explicit expression of Gauss sum for the special  linear group $SL_n(\mathbb{F}_q)$, and as a consequence, they obtained an improvement of Ferguson et al.'s result. More precisely, they showed that if $n=2$, then the condition $|E||F|\ge C q^5$ is enough, but in higher dimensional cases, we need $|E||F|\ge C q^{2n^2-2n}$. Note that a graph theoretic proof of the result for the case $n=2$ was given recently by Dem\.iro\u{g}lu Karabulut \cite{Yesim}. More precisely, she proved that if $|E||F|> 4q^{7}/(q-1)^2$, then  for every $t\in \mathbb F_q^*$ there exists $(x,y)\in E\times F$ such that $\det(x-y)=t.$ In Appendix, based on the discrete Fourier analysis, we will give an alternative proof for a similar result of Karabulut but for more accurate size conditions on sets: if $E, F\subseteq M_2(\mathbb F_q)$ with $|E||F|> 4q^5$, then we have $\det(E+F)\supseteq \mathbb F_q^*.$ 

We refer readers to \cite{covert, h1, h2, h3, koh, pp, R} for  recent results in the setting of matrix rings.
\subsection{Statement of main results} In this paper, we study Problem \ref{qs1} for $n=2$ through a discrete Fourier analysis based on an  Odot-product. 
For $i\in \mathbb F_q,$ recall that  $D_i$ is a subset of $ M_2(\mathbb F_q)$ defined  as 
$$ D_i= \{x\in M_2(\mathbb F_q): \det(x)=i\}.$$
For $S\subseteq M_2(\mathbb F_q)$, let $\det(S)$ denote the set of determinants generated by $S$, i.e.,
$$\det(S):=\{ \det(x)\in \mathbb F_q: x\in S\}.$$
 The first result of ours is concerned with the sum set  $S=E+F$ with a restriction $E\subseteq D_i$ and  $F\subseteq D_j$ for $i, j\in \mathbb{F}_q^*.$
Namely, we produce an optimal result on Problem \ref{qs1} for the sum set  $E+F$. 

\begin{theorem} \label{main1} For $i, j\in \mathbb F_q^*,$ let  $E\subseteq D_i$ and $F\subseteq D_j.$ 
If $|E||F|\ge {15}^2q^4$, then we have
$$ \det(E+F)=\mathbb F_q.$$
\end{theorem}

Note that this result should be compared with results of Ferguson et al., Li and Hu, and Karabulutin  in the paragraph subsequent to Problem \ref{qs1}.
In our result, we impose a stronger condition on $E, F$, i.e., $E\subseteq D_i$ and   $E\subseteq D_j$, than they did in \cite{shpa, li, Yesim} ,  while our threshold $q^4$ is much better than those in their results (for $n=2$).

\smallskip

One can easily construct an example to show that  the threshold $q^4$  can not be lower for arbitrary subsets $E,F$ of $M_2(\mathbb F_q)$.
For instance, let $q=p^2$ for some odd prime $p$ and take
$$E=F=\left\{\left(\begin{matrix}
x_1&x_2\\
x_3&x_4
\end{matrix}\right) \in M_2(\mathbb F_q): x_1,x_2,x_3,x_4 \in \mathbb F_p\right\}.$$ Then $|E|=|F|=q^2$ and $\det(E+F)=\mathbb F_p.$ This example proposes a conjecture that for any subsets $E,F$ of $M_2(\mathbb F_q)$ with $|E||F|\ge C q^4$ for a large constant $C>1,$ we have $\det(E+F)=\mathbb F_q.$ 
Notice that Theorem \ref{main1} confirms this conjecture (up to a constant) in the specific case when $E\subseteq D_i$ and $F\subseteq D_j$ for $i,j\ne 0.$  
Then there arises a natural question whether it is possible to improve the threshold $q^4$ in the specific cases.  In this paper we show that the threshold $q^4$ can not go lower in general, so Theorem \ref{main1} is sharp. Indeed, for any non-square number $i$ of $\mathbb F_q,$ we will construct a set $E\subseteq D_i$ such that  $|E|\sim q^2$, but $\det(E+E) \ne \mathbb F_q.$
\smallskip

Notice that we have obtained the very explicit constant $15^2$ for the bound in Theorem \ref{main1}. Such an explicit constant is  not available in the literature in general, and  is one of features this paper owns.  It would be interesting to search for  a smaller constant than this.\\

Taking $E=F$, the following corollary follows immediately from Theorem \ref{main1}. 
\begin{corollary}\label{maincor} Let $i$ be an element of $\mathbb{F}_q^*$ and  $E$ be a set in $D_i$. If $|E|\ge 15 q^{2},$ then we have 
$$ \det(E+E)= \mathbb F_q.$$
\end{corollary} 

As a motivation for the second result, let us first consider the following simple question to answer.
Given two varieties $D_i, D_j$ in $M_2(\mathbb F_q)$ for non-zero $i, j\in \mathbb F_q,$ determine the smallest exponent $\beta$ such that for any sets $E,F\subseteq M_2(\mathbb F_q)$ with  $|E||F|\ge C q^\beta$, we have
$$ \det\left((E\cap D_i)+(F\cap D_j)\right)=\mathbb F_q.$$
As it stands, the answer for the smallest exponent $\beta$ is $8.$
To see this,  take $E=M_2(\mathbb F_q)\setminus D_i$ and $F=M_2(\mathbb F_q)\setminus D_j$. Then $(E\cap D_i)+(F\cap D_j)$ is an empty set, and $|E|, |F|\sim q^4$ (equivalently, $|E||F|\sim q^8$). This example proposes that the smallest exponent $\beta$ can not be less than $8.$ On the other hand, if we take $E=F=M_2(\mathbb F_q)$, then $|E||F|=q^8$ and $\det\left((E\cap D_i)+(F\cap D_j)\right)=\mathbb F_q.$ 
  
However, in our second result we prove that  if we work with subsets $E,F$ with some restriction, we obtain a non-trivial result. To explain this, we fix the identification  $ M_2(\mathbb F_q)=\mathbb F_q^2\times \mathbb F_q^2$ through the assignment $$\left(\begin{matrix}
x_1&x_2\\
x_3&x_4
\end{matrix}\right)\mapsto \big(x_1,x_2, x_3,x_4\big).$$ 
 
Write $ \left(\begin{matrix} 
S_1\\
S_2
\end{matrix}\right)$ for the subset of $ M_2(\mathbb F_q)$ corresponding to $S_1\times S_2\subseteq  \mathbb F_q^2\times \mathbb F_q^2.$
We will say that  a subset $S\subseteq M_2(\mathbb F_q)$ is of product type if it is written as   $ S=\left(\begin{matrix} 
S_1\\
S_2
\end{matrix}\right)$ for some $S_1,S_2\subseteq \mathbb F_q^2.$
Then as an application of Theorem \ref{main1}  we obtain the following.
\begin{theorem} \label{Alexset}
Let $E, F\subseteq M_2(\mathbb F_q)$ be of product type. 
If $|E|, |F|\ge C q^3$ for a sufficiently large constant $C$, then for any $i, j\in \mathbb F_q^*,$ we have
$$ \det\left( (E\cap D_i) + (F\cap D_j)\right)=\mathbb F_q.$$
\end{theorem}
A few words on Theorem \ref{Alexset} are in order. First, note that the theorem implies that the subset $E\cap D_i$ is nonempty for any $i\ne 0$. In fact, we will see from Lemma \ref{bigcor} that   we have  $$|E\cap D_i|\sim \frac{|E|}{q},$$
which can be combined with  Theorem  \ref{main1} to deduce Theorem \ref{Alexset}. 
Also notice from Theorem \ref{Alexset} that  if $E=\left(\begin{matrix}
A&A\\
A&A
\end{matrix}\right)$ for some $A\subseteq \mathbb F_q$ with $|A|\ge C q^{3/4}$, then for any $i, j\ne 0,$ we have
$$ \det((E\cap D_i) + (E\cap D_j))=\mathbb F_q.$$ 

We now address an extension of Corollary \ref{maincor}. For a fixed $\ell \in \mathbb N$ and $E\subseteq M_2(\mathbb F_q),$ we define 
$$ \ell E=\underbrace{E + \dots + E}_{\ell~\textup{terms}}\subseteq M_2(\mathbb F_q).$$
In fact, in this case,  the threshold $q^2$ of Corollary \ref{maincor} can be improved whenever $\ell$ becomes larger as the following shows.
 
\begin{theorem}\label{thm0}
Let $k\ge 2$ be an integer and $i$ be an element in $\mathbb{F}_q^*$. If $E \subseteq D_i$ and $|E|\ge  Cq^{\frac{6k-5}{4k-4}}$ for a sufficiently large constant $C$,  then we have 
\[\det(2kE)=\det(\underbrace{E + \dots + E}_{2k~\textup{terms}})\supseteq \mathbb{F}_q^*.\]
\end{theorem}

It follows from Theorem \ref{thm0} that if $k$ is large enough, then $\det(2kE)\supseteq \mathbb{F}_q^*$ whenever the size of $E$ is close to $q^{\frac{3}{2}}$. 
However, in general, one can not expect to go lower than $q^{\frac{3}{2}}$. To see this, let $q=p^2$ for some odd prime $p$, and $E$ be the special linear group $SL_2(\mathbb{F}_p)$. Then it is obvious that  $|E|\sim p^{3}=q^{3/2}$. But since   $2kE$ is a subset of $M_2(\mathbb{F}_p)$ for any $k,$ we have $\det(2kE)\subseteq \mathbb{F}_p\subsetneq \mathbb{F}_q$.

\smallskip
Lastly, we would like to say a few words on the exposition of the paper. Unlike in the literature, we elaborated on finding  explicit constants $C$ for the bounds in Theorem \ref{main1} and Corollary \ref{maincor}. This asked us to write out almost all details for readers, which had the exposition  a bit lengthy, because they have their own distinctions and  some subtleties  even though some of them look similar.
  
\subsection{Outline of this paper} The remaining parts of this paper are organized to provide the complete proofs of our main theorems. In Section \ref{sec2N}, we summarize the background knowledge of the discrete Fourier analysis which will be used as a main tool. In particular, a new operation called the \textit{Odot-product} is introduced. Section \ref{sec3N} is designed to prove Theorem \ref{main1} whose sharpness is shown in Section \ref{sec4N}. In Section \ref{secEx}, a proof of Theorem \ref{Alexset} is given. In Section \ref{sec5N}, we obtain a lower bound on the cardinality of the sum of two matrix sets,
which will play a crucial role in proving Theorem \ref{thm0}. In the final section, we complete the proof of Theorem \ref{thm0}.
\subsection{Acknowledgement:} The authors would like to thank Igor Shparlinski  for introducing the paper of Li and Hu \cite{li} to them.

\section{Preliminaries}\label{sec2N}

In this section, we review the discrete Fourier analysis and exponential sums. In addition, we introduce the so-called   \textit{Odot-product} on $\MFq$ and investigate its properties which play a key role in proving our main results. 
\subsection{Discrete Fourier analysis and exponential sums} 
Throughout this paper, we will denote by $\chi: \mathbb F_q \to \mathbb S^1$ the canonical additive character of $\mathbb F_q.$ For instance, if $q$ is prime, then we have
$\chi(t)=e^{2\pi i t/q}$. If $q=p^n$ for some odd prime $p$, then we take
$\chi(t)=e^{2\pi i Tr(t)/p}$ for all $t\in \mathbb F_q,$
where $Tr$ denotes the trace function from $\mathbb F_q$ to $\mathbb F_p$ defined by
$$Tr(t)=t+t^p+ t^{p^2}+\cdots + t^{p^{n-1}} \in \mathbb F_p.$$ Recall that the character $\chi$ enjoys the orthogonality property; for any $m\in \mathbb F_q^d, ~d\ge 1$,
$$ \sum_{x\in \mathbb F_q^d} \chi(m\cdot x)=\left\{ \begin{array}{ll} 0  
\quad&\mbox{if}~~ m\neq (0,\dots,0)\\
q^d  \quad &\mbox{if} ~~m= (0,\dots,0), \end{array}\right.$$
where  $m\cdot x$ denotes the usual dot-product notation.
Given a complex-valued function $f$ defined on $\mathbb F_q^d,$ the Fourier transform of $f$ is defined by
$$ \widehat{f}(m):=q^{-d} \sum_{x\in \mathbb F_q^d} \chi(-m\cdot x) f(x).$$
The Plancherel theorem in this context says that
$$ \sum_{m\in \mathbb F_q^d} |\widehat{f}(m)|^2 = \frac{1}{q^d} \sum_{x\in \mathbb F_q^d} |f(x)|^2.$$
In particular, if $E \subseteq \mathbb F_q^d,$ then
$$ \sum_{m\in \mathbb F_q^d} |\widehat{E}(m)|^2 = \frac{|E|}{q^d}.$$
Here, throughout this paper, we identify the set $E \subseteq \mathbb F_q^d$ with the indicator function $1_E$ of the set $E.$\\

Let $\eta:\mathbb F_q^*\rightarrow  \mathbb S^1$ be the quadratic character of $\mathbb F_q^*$, i.e., a group homomorphism defined by $\eta(t)=1$ if $t$ is a square, and $-1$ otherwise. 
Recall that the orthogonality property of $\eta$ states that for any $a\in \mathbb F_q^*,$
$$ \sum_{t\in \mathbb F_q^*} \eta(at) =0.$$
Next, we collect  well-known properties of the Gauss sum and the Kloosterman sum. 
Let us begin by giving the definition of the Gauss sum.  The Gauss sum $G_a(\eta, \chi)$ associated with the characters $\chi$,  $\eta$, and an element $a\in \mathbb F_q^*$ is defined by
$$ G_a(\eta, \chi)=\sum_{t\in \mathbb F_q^*} \eta(t) \chi(at).$$
It is well known  that $|G_a(\eta, \chi))|=q^{1/2}$ for all $a\in \mathbb F_q^*.$ 
Moreover, the value of the Gauss sum for $a=1$ is explicitly given as follows. 
\begin{lemma}\cite[Theorem 5.15]{LN97}\label{ExplicitGauss}
Let ${\mathbb F}_q$ be a finite field with $ q= p^n$, where $p$ is an odd prime and $n \in {\mathbb N}.$
Then we have
$$G_1(\eta, \chi)=\left\{\begin{array}{ll}  {(-1)}^{n-1} q^{\frac{1}{2}} \quad &\mbox{if} \quad p \equiv 1 \mod 4 \\
                    {(-1)}^{n-1} i^n q^{\frac{1}{2}} \quad &\mbox{if} \quad p\equiv 3 \mod 4.\end{array}\right.$$
\end{lemma}
 
We notice that $\eta(-1)=1$ if and only if $-1$ is a square number of $\mathbb F_q$ (namely, $q\equiv 1 \mod 4$); or equivalently, $\eta(-1)=-1$ if and only if $-1$ is not a square number of $\mathbb F_q$ (namely, $q\equiv 3\mod 4$). From this fact and Lemma \ref{ExplicitGauss}, it follows that
\begin{equation}\label{Gauss-1} \eta(-1)G_1^2=q.\end{equation}

Hereafter, to use a simple notation, we write $G_1$ for $G_1(\eta, \chi).$\\

The following result is a corollary of Lemma 4.3 in \cite{HIKSU}.  For the reader's convenience, we provide a proof here.
\begin{lemma}\label{Comsquare} For $a\in \mathbb F_q^*, b\in \mathbb F_q,$ we have
$$ \sum_{s\in \mathbb F_q^*} \chi(as^2+bs) = \eta(a) G_1 ~\chi\left(\frac{b^2}{-4a}\right) -1.$$
\end{lemma}
 \begin{proof} Since $\chi(0)=1$, it is enough to prove that
 \begin{equation}\label{eqG} \sum_{s\in \mathbb F_q} \chi(as^2+bs)= \eta(a) G_1~\chi\left(\frac{b^2}{-4a}\right).\end{equation}
 Since $as^2+bs=a\left(s+ \frac{b}{2a}\right)^2-\frac{b^2}{4a},$ by a change of variables we have
 $$ \sum_{s\in \mathbb F_q} \chi(as^2+bs) = \sum_{s\in \mathbb F_q} \chi (as^2) \chi\left(\frac{b^2}{-4a}\right).$$
 Thus, the lemma follows from the observation that if $a\ne 0$, then
 $$ \sum_{s\in \mathbb F_q} \chi(as^2) = \eta(a) G_1.$$
 \end{proof}
We will also utilize the following properties of the Gauss sum which can be  proved by using a change of variables and properties of the quadratic character $\eta.$
For $a,b\ne 0$, we have
\begin{equation}\label{Gauss}\sum_{s\in \mathbb F_q^*} \eta(as) \chi(bs) =\sum_{s\in \mathbb F_q^*} \eta(as^{-1}) \chi(bs)= \eta(ab) G_1. 
\end{equation}


We review estimates on the (generalized) Kloosterman sum which can be found in \cite{IK04, LN97}.
An estimate of the Kloosterman sum  is given by
$$\left|\sum_{t\in {\mathbb F}_q^*}\chi(at+bt^{-1})\right| \le 2q^{\frac{1}{2}} \quad \mbox{for}\quad a,b \in {\mathbb F}_q^*,$$
and an estimate of the generalized Kloosterman sum is given by 
$$\left|\sum_{t\in {\mathbb F}_q^*}\eta(t)\chi(at+bt^{-1})\right| \le 2
q^{\frac{1}{2}} \quad \mbox{for} \quad a,b \in {\mathbb F}_q.$$

\subsection{Odot-product and its properties}

In this subsection, we will define
the so-called \textit{Odot-product} on the vector space $M_2(\mathbb F_q)=\mathbb F_q^4$, which can be compared with the ordinary inner product on $\mathbb F_q^4$. Then we will set up a main tool, i.e., a discrete Fourier theoretic machinery for the Odot-product, which is modeled on the well-established (discrete) one for the ordinary inner product. 

\begin{definition} [Odot-product]
 For $x=\left(\begin{matrix}
x_1&x_2\\
x_3&x_4
\end{matrix}\right),~
y=\left(\begin{matrix}
y_1&y_2\\
y_3&y_4
\end{matrix}\right) \in M_2(\mathbb F_q),$ 
define 
 \begin{equation}\label{Dodot} x\odot y:= x_1y_4-x_2y_3-x_3y_2+x_4y_1.\end{equation} 
Let us call $\odot$ the Odot-product on $ M_2(\mathbb F_q)$.
  \end{definition}
For $x\in M_2(\mathbb F_q)$, we will often use the notation $\|x\|_*$ to denote $\det(x).$ Namely, $$ \|x\|_*=\det(x)=x_1x_4-x_2x_3.$$
We collect basic properties of the Odot-product which follow easily from the definitions of the Odot-product and $\|\cdot\|_*.$   We leave the details to readers.
\begin{lemma} \label{Odotlem}
 Let $x,y\in M_2(\mathbb F_q), c\in \mathbb F_q.$ Then 
the Odot-product $\odot$ satisfies the followings.
 $$ x\odot y=y\odot x, \quad c(x\odot y)= (cx)\odot y=x\odot (cy), \quad \|cx\|_*=c^2\|x\|_*, $$
 $$x\odot x=2\|x\|_* \quad\mbox{and} \quad \|x\pm y\|_*=\|x\|_*+\|y\|_*\pm x\odot y.$$
 \end{lemma}
 One can check that the following orthogonality  of $\chi$ holds for the Odot-product: for $m\in M_2(\mathbb F_q),$
 \begin{equation}\label{odotorth} \sum_{x\in M_2(\mathbb F_q)} \chi(m\odot x)=\left\{\begin{array}{ll} 0&\quad\mbox{if}~~m\ne \mathbf{0},\\
                                                                      q^4&\quad\mbox{if}~~m=\mathbf{0}. \end{array} \right.\end{equation}
 Given a function $f:M_2(\mathbb F_q)  \to \mathbb C$, we define 
 \begin{equation}\label{DefFTi}\widetilde{f}(x):=\sum_{m\in M_2(\mathbb F_q)} \chi(-x\odot m) f(m).\end{equation}
For instance, for $i\in \mathbb F_q^*$ and $y\in M_2(\mathbb F_q)$, we have 
$$\widetilde{D_i} (y)=\sum_{x\in  D_i} \chi(-x\odot y) .$$

\begin{lemma}\label{Focirl}
For $i\in \mathbb F_q^*$ and $y\in M_2(\mathbb F_q)$, $\widetilde{D_i} (y)$ is expressed as
$$\widetilde{D_i} (y)   = q^3 \delta_0(y) + q\sum_{r\in \mathbb F_q^*} \chi\left(-ir-\frac{\|y\|_*}{r}\right),$$
where $\delta_0(y)=1$ if $y=\mathbf{0},$ and $0$ otherwise.
\end{lemma}
\begin{proof}
By the orthogonality  of $\chi$, we can write that  
$$ \widetilde{D_i} (y)=q^{-1}\sum_{x\in M_2(\mathbb F_q)} \sum_{r\in \mathbb F_q} \chi(r(\|x\|_* -i)) \chi(-x\odot y)$$
$$=q^3 \delta_0(y) + q^{-1}\sum_{x\in M_2(\mathbb F_q)} \sum_{r\ne 0} \chi(r(\|x\|_* -i)) \chi(-x\odot y)$$
$$= q^3\delta_0(y) + q^{-1} \sum_{r\ne 0} \chi(-ir) \sum_{x_1, x_2,x_3, x_4\in \mathbb F_q} \chi( rx_1x_4-rx_2x_3 - x_1y_4+x_2y_3+x_3y_2-x_4y_1).$$
Then the lemma follows from a calculation of the sums over $x_1, x_2\in \mathbb F_q$ using the orthogonality  of $\chi$.
\end{proof}

\section{The key lemma and proof of Theorem \ref{main1}} \label{sec3N}
This section is dedicated to proving Theorem \ref{main1}. We begin by introducing  notations for our interested quantities.
\smallskip

\begin{notation} Let $E, F$ be sets in $M_2(\mathbb F_q).$
\begin{enumerate}\item For $t\in \mathbb F_q$, we denote by $N_t(E,F)$  the number of pairs $(x,y)\in E\times F$ such that $\det(x+y)=t.$ 
\item For $\ell\in \mathbb F_q,$ we let $W_\ell(E,F)$ denote the number of pairs $(x,y)\in E\times F$ such that $x\odot y=\ell.$
\item For $\ell\in \mathbb F_q,$ we write $R_\ell(E,F)$ for $W_\ell(E,F)-|E||F|/q.$ Namely, 
$$ R_\ell(E,F):= W_\ell(E,F)-\frac{|E||F|}{q}.$$
\item We denote by $M(E,F)$ the maximum value of the set $\{W_\ell(E,F): \ell\in \mathbb F_q\}.$ Namely,
$$M(E,F):= \max_{\ell\in \mathbb F_q} \sum_{x\in E, y\in F: x\odot y=\ell} 1.$$ 
\end{enumerate}
\end{notation}
 A bound on $N_t(E,F)$  plays an essential role in proving Theorem \ref{main1}, as well as it is interesting on its own right.  To obtain an upper bound for $N_t(E,F),$  we need  a couple of technical lemmas. 

\begin{lemma}\label{twoC} For $i, j\in \mathbb F_q^*,$ let $E$ and $F$ be subsets of $D_i$ and $D_j$, respectively. Suppose that for all $\ell\in \mathbb F_q,$ the following two inequalities hold:
\begin{equation}\label{Hp}  |R_\ell(E,F)|^2 \le 2q^2|E||F|+ 7|E||F|^2+ 2q|E| M(F,F)\end{equation}
and
\begin{equation}\label{Hp1} |R_\ell(F,F)|^2 \le 2q^2|F|^2+ 7|F|^3+ 2q|F| M(F,F). \end{equation}
Then, for every $t\in \mathbb F_q$, we have
$$ \left|N_t(E,F)-\frac{|E||F|}{q}\right|\le \sqrt{18q^2|E||F| + 11|E||F|^2+ 4\sqrt{7}q|E||F|^{\frac{3}{2}}}.$$
\end{lemma}
\begin{proof} By definition, we can write
 $$ N_t(E,F)=\sum_{x\in E, y\in F: \|x+y\|_*=t} 1.$$
 Since $E\subseteq D_i, F\subseteq D_j,$ by Lemma \ref{Odotlem}, this can be written as
 $$ N_t(E,F)=\sum_{x\in E, y\in F: x\odot y=t-i-j} 1.$$
 Letting $\ell=t-i-j$, we see that $N_t(E,F)=W_\ell(E,F).$ Hence, to prove the lemma, it suffices to show that for all $\ell \in \mathbb F_q,$ we have
 $$ |R_\ell(E,F)|^2\le 18q^2|E||F| + 11|E||F|^2+ 4\sqrt{7}q|E||F|^{\frac{3}{2}}.$$
Notice from the assumption \eqref{Hp}  that to prove the above inequality it is enough to show that 
\begin{equation}\label{goalf} M(F,F)\le \frac{2|F|^2}{q}+8 q|F| + 2\sqrt{7} |F|^{3/2}.\end{equation}
Since $W_\ell(F,F)= |F|^2/q +R_\ell(F,F)$ by definition, it is clear that
$$M(F,F)= \frac{|F|^2}{q} + \max_{\ell\in \mathbb F_q} R_{\ell}(F,F),$$ 
and the assumption \eqref{Hp1} implies that
$$\max_{\ell\in \mathbb F_q} R_{\ell}(F,F)\le \sqrt{2} q|F| + \sqrt{7} |F|^{3/2} + \sqrt{2} q^{1/2} |F|^{1/2}M(F,F)^{1/2}.$$ 
Therefore, we have
\begin{align*}M(F,F)&\le \frac{|F|^2}{q} + \sqrt{2} q|F| + \sqrt{7} |F|^{3/2} + \sqrt{2} q^{1/2} |F|^{1/2}M(F,F)^{1/2}\\
& \le 2\max\left\{ \frac{|F|^2}{q}+\sqrt{2} q|F| + \sqrt{7} |F|^{3/2}, ~\sqrt{2} q^{1/2} |F|^{1/2}M(F,F)^{1/2}\right\}.\end{align*}
From this estimate, we obtain the inequality \eqref{goalf} as follows:
\begin{align*} M(F,F)&\le \max\left\{ \frac{2|F|^2}{q}+2\sqrt{2} q|F| + 2\sqrt{7} |F|^{3/2}, ~8q|F|\right\}\\
& \le \frac{2|F|^2}{q}+8 q|F| + 2\sqrt{7} |F|^{3/2}.\end{align*}
\end{proof}

As we will see, Proposition \ref{MainThm} given in the last part of this section plays a key role in proving Theorem \ref{main1}. Notice that the proof of Proposition \ref{MainThm} uses bounds of several summations. To make the exposition better, we separately treat these summations in several lemmas. 

\begin{lemma}\label{Lemma-A(l)}Let $F$ be a subset of $D_j$ with $j\in \mathbb F_q^*.$ Then,  for every $\ell \in \mathbb F_q,$ we have
 $$\mathcal{I}(\ell):=\sum_{\substack{y,y'\in F\\ s,s'\in {\mathbb F}_q^*}} \delta_0(s'y'-sy)
\chi\left(\ell(s'-s)\right)\le 2q|F|.$$  
\end{lemma}

\begin{proof}
The value $\mathcal{I}(\ell)$ can be written as
$$ \mathcal{I}(\ell)= \sum_{\substack{y,y'\in F\\ s,s'\in {\mathbb F}_q^*: s'y'=sy}} 
\chi\left(\ell(s'-s)\right).$$ 

It is clear that the sum over pairs $(s,s')$ with $s=s'$ is  $(q-1)|F|,$ and  the sum over pairs $(s,s')$ with $s\ne s'$ is 
$$  \sum_{\substack{y,y'\in F\\ s,s'\in {\mathbb F}_q^*: s'y'=sy, s\ne s'}} 
\chi\left(\ell(s'-s)\right) =  \sum_{\substack{y,y'\in F\\ a\ne 0, b\ne 0, 1: y'=by}} 
\chi\left(\ell a(1-b)\right),$$
where we use a change of variables by letting $a=s', b=s/s'.$

If $\ell\ne 0$, then this value is less than or equal to  zero, because the sum over $a\ne 0$ is $-1$ by the orthogonality of $\chi.$ If $\ell=0,$ then 
the value above is given by
$$  \sum_{\substack{y,y'\in F\\ a\ne 0, b\ne 0, 1: y'=by}} 1=(q-1) \sum_{\substack{y,y'\in F\\  b\ne 0, 1: y'=by}} 1.$$
Observe that if $b\ne 1$ and $y,y'\in D_j$ with $j\ne 0,$ then $y'=by$  only if $b=-1.$  
Thus, the value above is at most $(q-1)|F|.$
In summary, we  have proved that  for  any $\ell \in \mathbb F_q,$ 
$$ \mathcal{I}(\ell)\le 2(q-1)|F| \le 2q|F|,$$
as desired.

\end{proof}

\begin{lemma}\label{Lemma-B(l)} Let $i\in \mathbb F_q^*$ and  $F$ be a subset of $D_j$ with $j\in \mathbb F_q^*.$ Then, for all $\ell \in \mathbb F_q,$ we have
 $$\mathcal{A}(\ell):=\sum_{\substack{y,y'\in F\\ r,s,s'\in {\mathbb F}_q^*:\|s'y'-sy\|_*=0 }} 
 \chi\left(-ir\right) \chi\left(\ell(s'-s)\right)\le q|F|^2.$$
\end{lemma}

\begin{proof}
Since $i\ne 0$, the sum over $r\in \mathbb F_q^*$ of $\mathcal{A}(\ell)$ is -1. Thus, we have
$$ \mathcal{A}(\ell)=\sum_{\substack{y,y'\in F\\ s,s'\in {\mathbb F}_q^*:\|s'y'-sy\|_*=0 }} 
  -\chi\left(\ell(s'-s)\right). $$ 
Notice that $\mathcal{A}(\ell)$ is a real number since $\mathcal{A}(\ell)= \overline{\mathcal{A}(\ell)}.$ 
It is clear that the contribution of the case $s=s'$ to  $\mathcal{A}(\ell)$ is negative. Hence, 
$$ \mathcal{A}(\ell) \le \sum_{\substack{y,y'\in F\\ s,s'\in {\mathbb F}_q^*:\|s'y'-sy\|_*=0, s\ne s'}} 
  -\chi\left(\ell(s'-s)\right). $$ 
 Since  $F\subseteq D_j,$  the condition $\|s'y'-sy\|_*=0$ is equivalent to  $j s'^2 + js^2-s's(y'\odot y)=0.$ 
 Using a change of variables by letting $a=s', b=s/s'$, we have
 $$ \mathcal{A}(\ell) \le \sum_{\substack{y,y'\in F\\ a\ne 0, b\ne 0, 1: j+jb^2-b (y'\odot y)=0}} 
  -\chi\left(\ell a(1-b)\right).$$ 
  If $\ell=0$, then this value is obviously a non-positive real number. If $\ell\ne 0$, then  the sum over $a\ne 0$  is $-1.$ Hence, 
$$\mathcal{A}(\ell) \le \sum_{\substack{y,y'\in F\\ b\ne 0, 1: j+jb^2-b (y'\odot y)=0}}  1 \le q|F|^2,$$
as required.
 \end{proof}

\begin{lemma}\label{Lemma-B(21)} Let $F$ be a subset of $M_2(\mathbb F_q).$ Then, for every  $i\in \mathbb F_q^*,$ we have
$$\mathcal{B}(i):=\sum_{\substack{y,y'\in F\\r, s\in {\mathbb F}_q^*: \|y'-y\|_*\ne 0}} 
 \chi\left(-ir-\frac{s^2\|y'-y\|_*}{r}\right)\le 2q|F|^2.$$
 
\end{lemma}
\begin{proof}We apply Lemma \ref{Comsquare} with $b=0$ to get the following: 
 $$ \mathcal{B}(i)= G_1\sum_{\substack{y,y'\in F\\r\in {\mathbb F}_q^*: \|y'-y\|_*\ne 0}} 
 \chi(-ir)~ \eta\left(-\frac{\|y'-y\|_*}{r}\right) 
 -\sum_{\substack{y,y'\in F\\r\in {\mathbb F}_q^*: \|y'-y\|_*\ne 0}}  \chi(-ir)$$
 $$=G_1\sum_{\substack{y,y'\in F\\r\in {\mathbb F}_q^*: \|y'-y\|_*\ne 0}} 
 \chi(-ir)~ \eta\left(-\frac{\|y'-y\|_*}{r}\right) + \sum_{\substack{y,y'\in F\\: \|y'-y\|_*\ne 0}} 1.$$
 Since the sum over $r\in \mathbb F_q^*$ of the first term above is a Gauss sum,  it is easy to see that
 \begin{equation*} \mathcal{B}(i) \le  q|F|^2 + |F|^2 \le 2q|F|^2.\end{equation*}
\end{proof}

\begin{lemma} \label{Lemma-B(22)}
Let $i\in \mathbb F_q^*$ and $F$ be a subset of $D_j$ with $j\in \mathbb F_q^*.$ Then for all $\ell\in \mathbb F_q$, we have
$$\mathcal{C}(\ell):=\sum_{\substack{y,y'\in F\\r, s,s'\in {\mathbb F}_q^*: \|s'y'-sy\|_*\ne 0, s\ne s'}} 
 \chi\left(-ir-\frac{\|s'y'-sy\|_*}{r}\right) \chi\left(\ell(s'-s)\right).$$
 $$\le 2q|F|^2+  G_1 \sum_{\substack{y,y'\in F\\r\ne 0, b\ne 0}} \eta(-r)\chi(r(\ell^2-4ij)) ~\chi\left( r(\ell^2-4ij) b^2 +2r(2i(y'\odot y)-\ell^2)b\right).$$
\end{lemma}
\begin{proof}The value $\mathcal{C}(\ell)$ is rewritten as follows: 
 $$\mathcal{C}(\ell)= \sum_{\substack{y,y'\in F\\r, s,s'\in {\mathbb F}_q^*: \|y'-(s/s')y\|_*\ne 0, s/s'\ne 1}} 
 \chi\left(-ir-\frac{s'^2\|y'-(s/s')y\|_*}{r}\right) \chi\left(\ell s'(1-s/s')\right).$$
 By a change of variables with $a=s', b=s/s'$, we have 
 $$ \mathcal{C}(\ell)= \sum_{\substack{y,y'\in F\\r, a\ne 0, b\ne 0,1: \|y'-by\|_*\ne 0}} 
 \chi\left(-ir-\frac{a^2\|y'-by\|_*}{r}\right) \chi\left(\ell a(1-b)\right).$$
 Computing the sum over $a\in \mathbb F_q^*$ by Lemma \ref{Comsquare}, we have
 $$ \mathcal{C}(\ell)=G_1 \sum_{\substack{y,y'\in F\\r\ne 0, b\ne 0,1: \|y'-by\|_*\ne 0}} \eta\left(\frac{\|y'-by\|_*}{-r}\right) \chi(-ir) \chi\left(\frac{r\ell^2(b-1)^2}{4\|y'-by\|_*} \right)$$
 $$+\sum_{\substack{y,y'\in F\\r\ne 0, b\ne 0,1: \|y'-by\|_*\ne 0}} -\chi(-ir).$$
 In the first term we use a change of variables by replacing $r/(4\|y'-by\|_*)$ by $r$ and in the second term we compute the sum over $r\ne 0.$ Then we see that 
 $$ \mathcal{C}(\ell)=G_1 \sum_{\substack{y,y'\in F\\r\ne 0, b\ne 0,1: \|y'-by\|_*\ne 0}} \eta\left(\frac{1}{-4r}\right) \chi(-4ir \|y'-by\|_* ) \chi\left(r\ell^2(b-1)^2 \right)$$
 $$+  \sum_{\substack{y,y'\in F\\ b\ne 0,1: \|y'-by\|_*\ne 0}} 1.$$
 Since the second term above is less than $q|F|^2,$ it follows that
 $$ \mathcal{C}(\ell)\le q|F|^2+G_1 \sum_{\substack{y,y'\in F\\r\ne 0, b\ne 0,1: \|y'-by\|_*\ne 0}} \eta\left(\frac{1}{-4r}\right) \chi(-4i r \|y'-by\|_*) \chi\left(r\ell^2(b-1)^2 \right)$$
 $$ =q|F|^2+ G_1 \sum_{\substack{y,y'\in F\\r\ne 0, b\ne 0: \|y'-by\|_*\ne 0}} \eta\left(\frac{1}{-4r}\right) \chi(-4ir \|y'-by\|_* ) \chi\left(r\ell^2(b-1)^2 \right)$$
 $$-G_1 \sum_{\substack{y,y'\in F\\r\ne 0: \|y'-y\|_*\ne 0}} \eta\left(\frac{1}{-4r}\right) \chi(-4ir \|y'-y\|_* ). $$
Using the formula \eqref{Gauss} and the fact that $G_1^2$ is a real number with $G_1^2=\pm q$, we see that the third term above is a real number which is less than or equal to $q|F|^2.$ Hence, 
$$ \mathcal{C}(\ell)\le 2q|F|^2+G_1 \sum_{\substack{y,y'\in F\\r\ne 0, b\ne 0: \|y'-by\|_*\ne 0}} \eta\left(\frac{1}{-4r}\right) \chi(-4ir \|y'-by\|_* ) \chi\left(r\ell^2(b-1)^2 \right)$$
$$ =  2q|F|^2+ G_1 \sum_{\substack{y,y'\in F\\r\ne 0, b\ne 0}} \eta\left(\frac{1}{-4r}\right) \chi(-4ir \|y'-by\|_* ) \chi\left(r\ell^2(b-1)^2 \right)$$
$$- G_1 \sum_{\substack{y,y'\in F\\r\ne 0, b\ne 0: \|y'-by\|_*= 0}} \eta\left(\frac{1}{-4r}\right)  \chi\left(r\ell^2(b-1)^2 \right).$$
By the orthogonality of $\eta$, we see that if $\ell =0$ or $b=1$, then the last term above is zero. On the other hand, if $\ell\ne 0$ and $b\ne 1$, then it follows from the formula \eqref{Gauss} that the last term above is  
$$-\eta(-1)G_1^2 \sum_{\substack{y,y'\in F\\ b\ne 0: \|y'-by\|_*= 0}} 1.$$ This value is a negative real number since $\eta(-1)G_1^2=q$ (see \eqref{Gauss-1}). Hence,
$$  \mathcal{C}(\ell)\le   2q|F|^2+ G_1 \sum_{\substack{y,y'\in F\\r\ne 0, b\ne 0}} \eta\left(\frac{1}{-4r}\right) \chi(-4ir \|y'-by\|_*) \chi\left(r\ell^2(b-1)^2 \right).$$
Since $\|y'-by\|_*=j+jb^2-b(y'\odot y)$ for $b\in \mathbb F_q, y,y'\in F\subseteq D_j$, it follows that
\begin{align*}  \mathcal{C}(\ell)\le&~ 2q|F|^2\\ 
&+ G_1 \sum_{\substack{y,y'\in F\\r\ne 0, b\ne 0}} \eta(-r)\chi(r(\ell^2-4ij)) ~\chi\left( r(\ell^2-4ij) b^2 +2r(2i(y'\odot y)-\ell^2)b\right),\end{align*}
where we also used the fact that $\eta(1/(-4r))=\eta(-r).$ Thus, the proof is complete.
\end{proof}

Based on the previous lemmas, we can deduce the following result.
\begin{proposition}\label{MainThm}
Let $i,j$ be elements in $\mathbb F_q^*,$ and $E$ and $F$ be subsets of $D_i$ and $D_j,$ respectively. For each $t\in \mathbb F_q,$ we have
 $$ \left|N_t(E,F)-\frac{|E||F|}{q}\right|\le \sqrt{18q^2|E||F| + 11|E||F|^2+ 4\sqrt{7}q|E||F|^{\frac{3}{2}}}.$$
\end{proposition}

\begin{proof} To prove the proposition, we invoke Lemma \ref{twoC}, i.e., show that the conditions  \eqref{Hp} and \eqref{Hp1} in Lemma \ref{twoC} are satisfied.
 Note that if we prove the condition \eqref{Hp}, then we easily see that  the condition \eqref{Hp1} would be automatic by considering the case $E=F$. 
Thus it is enough to prove the  condition \eqref{Hp};  for each $\ell\in \mathbb F_q,$ 
\begin{equation}\label{maingoal} |R_\ell(E,F)|^2 \le 2q^2|E||F|+ 7|E||F|^2+ 2q|E| M(F,F).\end{equation}
To prove the above inequality, we first notice by the orthogonality  of $\chi$ that 
 $$W_\ell(E,F)=q^{-1} \sum_{x\in E, y\in F} \sum_{s\in \mathbb F_q} \chi(s(x\odot y-\ell))
 =\frac{|E||F|}{q} + q^{-1} \sum_{x\in E, y\in F}\sum_{s\ne 0}   \chi(sx\odot y)\chi( -s\ell).$$ 
From this equality, we see that  
\begin{equation*}\label{Rl}R_\ell(E,F)= q^{-1} \sum_{x\in E, y\in F}\sum_{s\ne 0}   \chi(sx\odot y)\chi( -s\ell).\end{equation*}
By the Cauchy-Schwarz inequality 
w.r.t $x\in E, $ we have
 \begin{align*}
|R_\ell(E,F)|^2 \leq  q^{-2} |E| \sum_{x\in E} \left| \sum_{y\in F, s\in {\mathbb F}_q^*}\chi(sx\odot y)\chi( -s\ell)\right|^2.
\end{align*}
Since $E\subseteq D_i$,  we have
\begin{align*}
|R_\ell(E,F)|^2 &\leq q^{-2} |E|\sum_{x\in D_i} \sum_{\substack{y,y'\in F\\ s,s'\in {\mathbb F}_q^*}} \chi\left(-x\odot (s'y'-sy)\right)
\chi\left(\ell(s'-s)\right)\\
&=q^{-2} |E| \sum_{\substack{y,y'\in F\\ s,s'\in {\mathbb F}_q^*}} \widetilde{D_i}(s'y'-sy)
\chi\left(\ell(s'-s)\right).
\end{align*}
Using Lemma \ref{Focirl} and Lemma \ref{Lemma-A(l)}, 
$$|R_\ell(E,F)|^2\le 2q^2|E||F|+ q^{-1}|E|\sum_{\substack{y,y'\in F\\ s,s'\in {\mathbb F}_q^*}} 
\sum_{r\in \mathbb F_q^*} \chi\left(-ir-\frac{\|s'y'-sy\|_*}{r}\right) \chi\left(\ell(s'-s)\right).$$
Let $A(\ell)$ denote the second term of the RHS of the above inequality. 
Then, to prove the inequality \eqref{maingoal}, it is enough to show that
\begin{equation*}\label{maingoal2} A(\ell) \le   7|E||F|^2+ 2q|E| M(F,F).\end{equation*}
To prove this inequality, we split up the sum $A(\ell)$ into two summands as follows:
$$ A(\ell)= q^{-1}|E|\sum_{\substack{y,y'\in F\\ r,s,s'\in {\mathbb F}_q^*:\|s'y'-sy\|_*=0 }} 
 \chi\left(-ir\right) \chi\left(\ell(s'-s)\right) $$ 
 $$+ q^{-1}|E|\sum_{\substack{y,y'\in F\\r, s,s'\in {\mathbb F}_q^*: \|s'y'-sy\|_*\ne 0}} 
 \chi\left(-ir-\frac{\|s'y'-sy\|_*}{r}\right) \chi\left(\ell(s'-s)\right).$$
From Lemma \ref{Lemma-B(l)}, it is clear that the first term of the  RHS of the above equality is $\le |E||F|^2.$ Hence, letting $B(\ell)$ denote the second term of the RHS of the above equality, we only need to show that 
$$ B(\ell) \le 6|E||F|^2+ 2q|E| M(F,F).$$
To estimate $B(\ell),$ we consider two cases that $s= s'$ and $s\ne s'.$ It follows that
  $$ B(\ell)= q^{-1}|E|\sum_{\substack{y,y'\in F\\r, s\in {\mathbb F}_q^*: \|y'-y\|_*\ne 0}} 
 \chi\left(-ir-\frac{s^2\|y'-y\|_*}{r}\right) $$
 $$+q^{-1}|E|\sum_{\substack{y,y'\in F\\r, s,s'\in {\mathbb F}_q^*: \|s'y'-sy\|_*\ne 0, s\ne s'}} 
 \chi\left(-ir-\frac{\|s'y'-sy\|_*}{r}\right) \chi\left(\ell(s'-s)\right).$$
It is obvious from Lemma \ref{Lemma-B(21)} that the first term of the  RHS of the above equality is 
$\le 2|E||F|^2.$ Therefore, letting $C(\ell)$ be the second term of the RHS of the above equality, our problem is reduced to showing that 
$$ C(\ell) \le 4|E||F|^2+ 2q|E| M(F,F).$$
Using Lemma \ref{Lemma-B(22)}, it follows that
$$ C(\ell)\le 2|E||F|^2 + G_1 q^{-1}|E| \sum_{\substack{y,y'\in F\\r\ne 0, b\ne 0}} \eta(-r)\chi(r(\ell^2-4ij)) ~\chi\left( r(\ell^2-4ij) b^2 +2r(2i(y'\odot y)-\ell^2)b\right).$$
Letting $D(\ell)$ denote the second term of the RHS of the above inequality, it is enough to prove that
\begin{equation}\label{maineq1} D(\ell) \le  2|E||F|^2+ 2q|E| M(F,F).\end{equation}
When $\ell^2-4ij=0,$ it is not hard to see that $D(\ell)=0.$ Thus, assuming that $\ell^2-4ij\ne 0$, we will prove the inequality \eqref{maineq1}. Computing the sum over $b\ne 0$ of the term $D(\ell)$ by using Lemma \ref{Comsquare}, we have
\begin{align} \label{D(l)eq} D(\ell)=&q^{-1}|E|G_1^2 \sum_{\substack{y,y'\in F\\r\ne 0}} \eta(4ij-\ell^2) \chi\left(\frac{\left[ (\ell^2-4ij)^2-(2i(y'\odot y)-\ell^2)^2 \right]r}{\ell^2-4ij}\right)\\
\nonumber &-q^{-1}|E|G_1 \sum_{\substack{y,y'\in F\\r\ne 0}} \eta(-r)\chi(r(\ell^2-4ij)).  \end{align}
The last value above is the same as
$$-q^{-1}|E|G_1^2 \sum_{y,y'\in F} \eta(4ij-\ell^2)$$
which is clearly $\le |E||F|^2.$ Hence, letting $F(\ell)$ be the first term of the RHS of the above equality \eqref{D(l)eq}, our final task is to show that
\begin{equation}\label{Great}F(\ell)\le |E||F|^2+ 2q|E| M(F,F).\end{equation}

Notice that the value in the bracket $[\quad]$ in \eqref{D(l)eq} is zero if and only if $y'\odot y=2j$ or $(\ell^2-2ij)/i.$ 
Hence, in the case of  $y'\odot y\ne 2j, (\ell^2-2ij)/i,$  the contribution to $F(\ell)$ is at most $|E||F|^2$, because the sum over $r\ne 0$ is $-1, G_1^2= \pm q,$ and $\eta$ takes $\pm 1.$
On the other hand, in the case of $ y'\odot y=2j$ or $(\ell^2-2ij)/i$, the contribution to $F(\ell)$ is clearly dominated by
$$ 2q|E| \max_{k\in \mathbb F_q}\sum_{\substack{y,y'\in F: y'\odot y=k }} 1.$$
Thus, the inequality \eqref{Great} holds and the proof of the proposition is complete.
\end{proof}
\smallskip
\subsection{Proof of Theorem \ref{main1}}
In this subsection, we give a proof of Theorem \ref{main1} for which we heavily use Proposition \ref{MainThm}.
\begin{proof} Note that  the hypothesis $|E||F|\ge {15}^2q^4$ implies that  $|E|\geq 15q^2$ or   $|F|\geq 15q^2$, say that $|E|\geq 15q^2$. Then we see that $|E|^{1/2}|F|^{1/2} \ge 15q^2$ and $|E|^{1/2}\ge \sqrt{15} q.$ This clearly implies that 
\begin{equation}\label{sizecom} |E||F|^{\frac{1}{2}} \ge 15\sqrt{15}~ q^3.\end{equation} 
In view of  Proposition \ref{MainThm}, it suffices to prove that  if $|E||F|\ge {15}^2q^4$ and $|E|\ge 15q^2,$ then
\begin{equation}\label{inequality_EF} \frac{|E||F|}{q} > \sqrt{18q^2|E||F| + 11|E||F|^2+ 4\sqrt{7}q|E||F|^{\frac{3}{2}}}.\end{equation}
By squaring both sides of (\ref{inequality_EF}) and simplifying it, we see that, to obtain the inequality (\ref{inequality_EF}), it is enough to show that
\begin{equation}\label{inequality_EF2} |E||F|> 18q^4+11q^2|F|+4\sqrt{7}q^3 |F|^{\frac{1}{2}}.\end{equation}
Since $|E|\ge 15q^2,$ and hence $|E||F|=\frac{11}{15}|E||F| + \frac{4}{15}|E||F|$ and $\frac{11}{15}|E||F|\ge 11q^2|F|$, for the inequality (\ref{inequality_EF2}) it is enough to prove that
\begin{equation}\label{inequality_EF3}\frac{4}{15}|E||F| > 18q^4+4\sqrt{7}q^3 |F|^{\frac{1}{2}}.\end{equation}
Write
$$\frac{4}{15}|E||F|= \frac{2}{25}|E||F| + \frac{14}{75}|E||F|.$$ Then the inequality (\ref{inequality_EF3}) would follow if we show two inequalities; 
\begin{equation}\label{A1} \frac{2}{25}|E||F|\ge 18q^4 \end{equation}
and
\begin{equation} \label{A2} \frac{14}{75}|E||F| > 4\sqrt{7}q^3 |F|^{\frac{1}{2}}.\end{equation}
The inequality \eqref{A1} follows immediately from our assumption that $|E||F|\ge {15}^2 q^4.$ The inequality \eqref{A2}  is equivalent to 
$$ |E||F|^{\frac{1}{2}} > \frac{150\sqrt{7}}{7}~ q^3,$$ which is immediate from \eqref{sizecom}. This proves the theorem.
\end{proof}

\section{Sharpness of Theorem \ref{main1}}\label{sec4N}
In this section, 
we will show that by giving a concrete example, Theorem \ref{main1} can not be  improved in general.   Let $H$ be a subvariety of $\MFq$ defined by the equation $x_2+x_3=0,$ and $H_i:=H\cap D_i$, so that we have
\begin{equation}\label{defH} H_i=\left\{\left(\begin{matrix}
x_1&x_2\\
-x_2&x_4
\end{matrix}\right) \in M_2(\mathbb F_q): x_1x_4+x_2^2=i \right\}.\end{equation}
 Then it is clear that $|H_i|\sim q^2$, and $-x\in H_i$ if and only if $x\in H_i.$
Let  $E$ be a maximal subset of $H_i$ such that $E\cap (-E)=\phi.$ Then it is obvious that
\begin{equation}\label{defE} |E|\sim q^2.\end{equation}

\begin{proposition}\label{prop:solution}Let $i$ be a non-square number in $\mathbb F_q^*$, and let $E$ be  a subset of $H_i$ given as in the above. Fix $y\in E.$ Then the equation for $x$; $x\odot y=-2i$  has a unique solution $x=-y$ in $E.$
\end{proposition}
A proof of Proposition \ref{prop:solution} will be given shortly after a proof of Corollary \ref{counterexample} below. The following  indicates that Theorem \ref{main1} is sharp in general.
\begin{corollary} \label{counterexample} Let $i$ and $E$ be  given as in Proposition \ref{prop:solution}.  Then we have
$$ 0\notin \det(E+E).$$
\end{corollary}
\begin{proof} Since $\det(x+y)=\|x+y\|_*=2i+x\odot y$ for $x,y \in E\subset D_i,$ it suffices to show that 
$$ x\odot y\ne -2i\quad \mbox{for any}~~x,y\in E.$$
Let us  assume that $x\odot y=-2i$ for some $x,y\in E.$ Then by Proposition \ref{prop:solution}, we have the relation $y=-x,$ so $-E\cap E$ is not empty.
However, this is impossible by the condition on $E.$  This proves the corollary.
\end{proof}
\subsection{Proof of Proposition \ref{prop:solution}}
\begin{proof}[Proof of Proposition \ref{prop:solution}]
It is obvious that $-y\odot y=-2i,$ so $x=-y$ is a solution to $x\odot y=-2i$.
Let us show the uniqueness. The conditions $x\in E$,  $y\in E$,  $x\odot y=-2i$, respectively, turn into 
 $$x_1x_4+x_2^2=i,$$ $$y_1y_4+y_2^2=i,$$ $$ (x_1, x_2, -x_2, x_4) \odot (y_1, y_2, -y_2, y_4)= -2i.$$ 

Let $N$ denote the number of solutions to the above equations for $x_1, x_2, x_4.$ We aim to prove that $N=1.$ 
Since $(x_1, x_2, -x_2, x_4) \odot (y_1, y_2, -y_2, y_4)=y_4x_1+2y_2x_2+y_1x_4,$ 
we can write
$$ N=\sum_{\substack{x_1, x_2, x_4\in \mathbb F_q:\\
 y_4x_1+2y_2x_2+y_1x_4=-2i,\\ x_1x_4+x_2^2=i }} 1.$$
 By the orthogonality  of $\chi$, we have
 $$ N=q^{-2} \sum_{x_1, x_2, x_4\in \mathbb F_q}\sum_{s, r\in \mathbb F_q} \chi(s(y_4x_1+2y_2x_2+y_1x_4+2i))
\chi(r(x_1x_4+x_2^2-i)).$$
Decomposing the `internal' sum $\sum_{s,r\in \mathbb F_q}$ into four summands
$$\sum_{s,r\in \mathbb F_q} = \sum_{s=0, r=0} + \sum_{s=0, r\ne 0}+ \sum_{s\ne 0, r=0}+\sum_{s\ne 0,r\ne 0},$$  we obtain four corresponding summands of $N$(in order) $$N=N_1+N_2+N_3+N_4.$$ 
Now we calculate $N_i$s.
First of all, $N_1$ is computed:
\begin{equation}\label{N1} N_1=q^{-2}\sum_{x_1, x_2, x_4\in \mathbb F_q} 1 = q.\end{equation}
Secondly, $N_2$ is given as follows:
$$ N_2=q^{-2} \sum_{x_1, x_2, x_4\in \mathbb F_q}\sum_{r\ne 0}\chi(r(x_1x_4+x_2^2-i))$$
\begin{equation}\label{Equal-N}=q^{-2}\sum_{r\ne 0} \chi(-ir) \left(\sum_{x_2\in \mathbb F_q} \chi(rx_2^2)\right) \left(\sum_{x_1, x_4\in \mathbb F_q} \chi(rx_1x_4)\right).\end{equation}
In (\ref{Equal-N}), the sum over $x_2\in \mathbb F_q$  is equal to $\eta(r) G_1$ by Lemma \ref{Comsquare}, and the one over $x_1,x_4\in \mathbb F_q$ is equal to $q$ by the orthogonality  of $\chi.$ Therefore, we see that
$$ N_2=q^{-1} G_1 \sum_{r\ne 0} \eta(r) \chi(-ri).$$
Now, by the formula in \eqref{Gauss}, we have
\begin{equation}\label{N2} N_2=q^{-1} G_1^2 \eta(-i).\end{equation}
Thirdly, the term $N_3$ is given as follows:
$$N_3=q^{-2} \sum_{x_1, x_2, x_4\in \mathbb F_q}\sum_{s\ne 0} \chi(s(y_4x_1+2y_2x_2+y_1x_4+2i)).$$
Since $y_1y_4+y_2^2=i\ne 0,$  one of $y_i, i=1,2,3,$ is not a zero. Then the orthogonality of $\chi$ yields that
\begin{equation}\label{N3} N_3=0.\end{equation}
Lastly, the term $N_4$ is written as follows:
$$ N_4=q^{-2} \sum_{x_1, x_2, x_4\in \mathbb F_q}\sum_{s\ne 0, r\ne 0} \chi(s(y_4x_1+2y_2x_2+y_1x_4+2i))
\chi(r(x_1x_4+x_2^2-i))$$
\begin{equation}\label{Equal-N-4}=q^{-2}\sum_{s\ne 0, r\ne 0} \chi(2is)\chi(-ir) \sum_{x_4\in \mathbb F_q} \chi(sy_1x_4) 
\left(\sum_{x_1\in \mathbb F_q} \chi((sy_4+ rx_4)x_1) \right)\left(\sum_{x_2\in\mathbb F_q} \chi(rx_2^2+2sy_2x_2) \right).\end{equation}
In the term (\ref{Equal-N-4}), by the orthogonality  of $\chi$, the sum over $x_1\in \mathbb F_q$ is equal to $q$ if $x_4=-sy_4/r$, and $0$ otherwise. By the formula \eqref{eqG}, the sum over $x_2\in \mathbb F_q$ is equal to 
$$ \eta(r) G_1 \chi\left(\frac{s^2y_2^2}{-r}\right).$$
It follows that
$$ N_4= q^{-1}G_1 \sum_{s\ne 0, r\ne 0}\eta(r) \chi(-ir)\chi(2is) \chi\left(\frac{(y_1y_4+y_2^2)s^2}{-r}\right).$$
Since $y_1y_4+y_2^2=i,$ $N_4$ is written as
\begin{equation}\label{Equal-2} N_4=q^{-1}G_1 \sum_{r\ne 0} \eta(r) \chi(-ir) \left(\sum_{s\ne 0} \chi\left(\frac{is^2}{-r} +2is\right)\right).\end{equation}
Using Lemma \ref{Comsquare} to compute the sum over $s\ne 0$ in (\ref{Equal-2}), we obtain that
$$N_4=q^{-1}G_1^2 \sum_{r\ne 0}\eta(-i) -q^{-1}G_1 \sum_{r\ne 0} \eta(r)\chi(-ir)$$
$$ = q^{-1}G_1^2\eta(-i) (q-1) -q^{-1}G_1^2 \eta(-i)=G_1^2\eta(-i)-2q^{-1}G_1^2 \eta(-i).$$ Adding all $N_i$ with $1\leq i \leq 4$,  we obtain
$$ N=N_1+N_2+N_3+N_4=q+\eta(i) G_1^2 \eta(-1) -\eta(i) q^{-1} G_1^2 \eta(-1).$$
Since $i$ is a non-square number,  $\eta(i)=-1$. Recall from \eqref{Gauss-1} that $G_1^2 \eta(-1)=q.$
Thus $N=1,$ as required. This completes the proof of Lemma \ref{counterexample}.
\end{proof}

\section{Proof of Theorem \ref{Alexset}}\label{secEx}

In this section we prove Theorem \ref{Alexset} by using Theorem \ref{main1} and a result on the size of the intersection of a product type subset $S$ and $D_i$  with $i\ne 0.$ For the latter result, we estimate $|S\cap D_i|$ by adapting the method which Hart and Iosevich \cite{HI08} used in studying the size of the dot-product set determined by a set in $\mathbb F_q^d.$

\begin{lemma}\label{bigcor} Let $S\subseteq M_2(\mathbb F_q)$ be of product type. Then, for each $i\in \mathbb F_q^*,$ we have
$$ \left| |S\cap D_i|-\frac{|S|}{q}\right| \le q^{1/2} |S|^{1/2}.$$
\end{lemma}

\begin{proof} Let $S=\left(\begin{matrix} 
S_1\\
S_2
\end{matrix}\right)$ for some $S_1,S_2\subseteq \mathbb F_q^2.$ It is clear that $|S|=|S_1||S_2|.$ It follows that 
$$ |S\cap D_i|= \sum_{\alpha\in S_1, \beta\in S_2: \det(\alpha, \beta)=i} 1,$$
where $\det(\alpha, \beta):=\alpha_1\beta_2-\alpha_2\beta_1$ for $\alpha=(\alpha_1, \alpha_2), \beta=(\beta_1, \beta_2)\in \mathbb F_q^2.$ By the orthogonality of $\chi$, we have
\begin{align*}\label{remaindot} |S\cap D_i|&= q^{-1} \sum_{\alpha\in S_1, \beta\in S_2} \sum_{r\in \mathbb F_q} \chi(r(\det(\alpha, \beta)-i))\\
&= \frac{|S_1||S_2|}{q} +  q^{-1} \sum_{\alpha\in S_1, \beta\in S_2} \sum_{r\ne 0} \chi(r(\det(\alpha, \beta)-i))\\
&:= \frac{|S|}{q} + R(i).\end{align*}
Hence, in order to prove the lemma, it will be enough to show that 
$$ |R(i)|^2\le q|S|.$$
Now, applying the Cauchy-Schwarz inequality to $|R(i)|^2$ w.r.t $\alpha\in S_1,$ and then replacing the index set ``$\alpha\in S_1$" by ``$\alpha\in \mathbb F_q^2$", we see 
$$ |R(i)|^2 \le q^{-2} \left(\sum_{\alpha\in S_1} \left|\sum_{\beta\in S_2, r\ne 0} \chi(r(\det(\alpha, \beta)-i))\right|\right)^2\le q^{-2} |S_1| \sum_{\alpha\in \mathbb F_q^2} \left|\sum_{\beta\in S_2, r\ne 0} \chi(r(\det(\alpha, \beta)-i))\right|^2.$$
Note that the rightmost term of this inequality is in turn equal to 
$$q^{-2} |S_1| \sum_{\alpha\in \mathbb F_q^2} \sum_{\beta, \beta'\in S_2, r,r'\ne 0} \chi(i(r'-r)) \chi(r\det(\alpha, \beta)-r'\det(\alpha, \beta^\prime)).$$
Next, we compute the sum over $\alpha\in \mathbb F_q^2$ by using the orthogonality of $\chi$ and  obtain
$$ |R(i)|^2 \le |S_1| \sum_{\beta, \beta^\prime\in S_2, r, r'\ne 0: r\beta=r'\beta^\prime} \chi(i(r'-r)).$$
Considering the cases that $r=r'$ and $r\ne r'$, we have
$$|R(i)|^2 \le |S_1| \sum_{\beta\in S_2, r\ne 0} 1 + |S_1| \sum_{\substack{\beta, \beta^\prime\in S_2, r,r'\ne 0:\\r\ne r', r\beta=r'\beta^\prime}} \chi(i(r'-r))$$
$$\le q|S_1||S_2| + |S_1|  \sum_{\substack{\beta, \beta^\prime\in S_2, r,r'\ne 0:\\r/r'\ne 1, (r/r')\beta=\beta^\prime}} \chi(i r'(1-r/r')).$$
By a change of variables with $a=r', ~b=r/r'$, 
\begin{equation}\label{eqRi} |R(i)|^2\leq q|S| + |S_1|\sum_{\substack{\beta, \beta^\prime\in S_2, a\ne 0, b\ne 0,1:\\ b\beta=\beta^\prime}} \chi(i a(1-b)).\end{equation}
The second term  in RHS of the inequality \eqref{eqRi} is non-positive, because the sum over $a\ne 0$ is -1 by the orthogonality of $\chi.$ Hence, we obtain $|R(i)|^2\le q|S|,$ as required.
\end{proof}

\smallskip

\subsection{Proof of Theorem \ref{Alexset}}
 \begin{proof}[Proof of Theorem \ref{Alexset}]
Since $|E|,|F| \ge C q^3,$ we see from Lemma \ref{bigcor} that
$ |E\cap D_i|\sim |E|/q$ and $|F\cap D_j| \sim |F|/q.$ Since $(E\cap D_i)\subseteq D_i,$ and $(F\cap D_j)\subseteq D_j,$  
the theorem follows from Theorem \ref{main1}.
\end{proof}

\section{Sum of two matrix sets}\label{sec5N}
For $E, F\subseteq M_2(\mathbb F_q)$, the sum set $E+F$ is defined by
$$ E+F:=\{x+y\in M_2(\mathbb F_q): x\in E, y\in F\}.$$
In this section, we shall give a `general' lower bound for sizes of sets $E+F$ when $E$ and $F$ are subsets  $ D_i$ and $D_j$ for nonzero $i,j\in \mathbb F_q.$  This result is one of main ingredients of the proof of Theorem \ref{thm0} given in the next section.

Recall that $N_t(E,F)$ denotes the number of pairs $(x,y)\in E\times F$ such that $\det(x+y)=t$.
\begin{lemma}\label{secondlem} 
If $E\subseteq D_i,~F\subseteq D_j$ for nonzero $i,j\in \mathbb F_q$, then we have
$$ \max_{t\in \mathbb F_q} N_t(E,F) \ll \frac{|E||F|}{q} + q |E|^{\frac{1}{2}}|F|^{\frac{1}{2}}.$$
\end{lemma} 
\begin{proof}
From Proposition \ref{MainThm}, we have
$$ \max_{t\in \mathbb F_q} N_t(E,F)\le  \frac{|E||F|}{q} + \sqrt{18q^2|E||F| + 11|E||F|^2+ 4\sqrt{7}q|E||F|^{\frac{3}{2}}}.$$

Using the basic fact that $ \sqrt{a+b} \le \sqrt{a}+\sqrt{b}$ for $a,b\ge 0,$ we obtain the estimate:
\begin{equation}\label{e11}
\max_{t\in \mathbb F_q} N_t(E,F) \ll \frac{|E||F|}{q} + q |E|^{\frac{1}{2}} |F|^{\frac{1}{2}} + |E|^{\frac{1}{2}} |F| + q^{\frac{1}{2}} |E|^{\frac{1}{2}} |F|^{\frac{3}{4}}.
\end{equation}
Switching roles of $E$ and $F$ in \eqref{e11}, we also obtain
\begin{equation}\label{e22} 
\max_{t\in \mathbb F_q} N_t(E,F) \ll \frac{|E||F|}{q} + q |E|^{\frac{1}{2}} |F|^{\frac{1}{2}} + |E||F|^{\frac{1}{2}} + q^{\frac{1}{2}} |E|^{\frac{3}{4}}|F|^{\frac{1}{2}}.
\end{equation}

For $1\leq r\leq 4$, let $a_r$ be the $r$-th term in RHS of (\ref{e11}), and for $r=3,4$,
$a_r^\prime$ the $r$-th term in RHS of (\ref{e22}):
$$a_1+a_2+a_3+a_4=\frac{|E||F|}{q} + q |E|^{\frac{1}{2}} |F|^{\frac{1}{2}} + |E|^{\frac{1}{2}} |F| + q^{\frac{1}{2}} |E|^{\frac{1}{2}} |F|^{\frac{3}{4}}$$

$$a_1+a_2+a_3^\prime+a_4^\prime=\frac{|E||F|}{q} + q |E|^{\frac{1}{2}} |F|^{\frac{1}{2}} + |E||F|^{\frac{1}{2}} + q^{\frac{1}{2}} |E|^{\frac{3}{4}}|F|^{\frac{1}{2}}.$$

To prove the lemma, we consider two cases.
\smallskip

\textbf{Case 1:} Assume that $|E|\le q^2$ or $|F|\le q^2.$ 
Indeed, if $|E|\le q^2$, then it follows from \eqref{e22} that $a_2\ge a_3^\prime$ and $a_2\ge a_4^\prime.$ If $|F|\le q^2$, then we see from \eqref{e11} that $a_2\ge a_3$ and $a_2\ge a_4.$ Thus, in this case, we have 
$$\max_{t\in \mathbb F_q} N_t(E,F) \ll a_1+a_2.$$

\noindent\textbf{Case 2:} Assume that $|E|>q^2$ and $|F|>q^2.$ 
It follows from \eqref{e11} that $a_1> a_3$ and $a_1> a_4.$ Hence,  in this case we also have
$$\max_{t\in \mathbb F_q} N_t(E,F)\ll a_1+a_2.$$

This completes the proof.
\end{proof}

Recall that  $W_\ell(E,F)$ denotes the number of pairs $(x,y)\in E\times F$ such that $x\odot y=\ell$. Note that if $E\subseteq D_i$ and $ F\subseteq D_j$ for some $i,j\in \mathbb F_q$, then   we have $W_\ell(E,F)=N_t(E,F),$ where $t=\ell+i+j$. Thus Lemma \ref{secondlem} can be restated as follows.

\begin{corollary}\label{easyCor}
Let $E, F$ be the sets given in Lemma \ref{secondlem}. Then we have
$$ \max_{\ell\in \mathbb F_q} W_\ell(E,F) \ll \frac{|E||F|}{q} + q |E|^{\frac{1}{2}}|F|^{\frac{1}{2}}.$$
\end{corollary}

For any  $E, F\subseteq M_2(\mathbb F_q)$, not necessarily contained in $ D_i$ for some $i$,  we produce an upper bound of $ W_0(E,F)$, which  will be also used in proving the main result of this section.
\begin{lemma} \label{odotzero}
Let $E, F\subseteq M_2(\mathbb F_q).$ Then we have
\begin{equation}\label{zeroK} W_0(E,F)\le \frac{|E||F|}{q} + \sqrt{2} q^{2} |E|^{\frac{1}{2}}|F|^{\frac{1}{2}}.\end{equation}
\end{lemma}
\begin{proof} We proceed as in the proof of  Lemma \ref{bigcor}. By the orthogonality  of $\chi,$  we can write  
$$ W_0(E,F)= \sum_{x\in E, y\in F: x\odot y=0} 1 = \frac{|E||F|}{q} + q^{-1} \sum_{x\in E, y\in F, s\ne 0} \chi(s(x\odot y)).$$ Let $$\Omega:=q^{-1} \sum_{x\in E, y\in F, s\ne 0} \chi(s(x\odot y)).$$
Notice that to complete the proof of the lemma, it suffices to prove that
$$ |\Omega|^2 \le 2q^4 |E||F|.$$ 
Let us bound $|\Omega|^2.$ 
First, applying the Cauchy-Schwarz inequality to $|\Omega|^2$ w.r.t $x\in E,$ and next replacing the index set ``$x\in E$" by ``$x\in M_2(\mathbb F_q)$", we obtain 
$$|\Omega|^2 
\le q^{-2} |E| \sum_{x\in M_2(\mathbb F_q)} \left|\sum_{y\in F, s\ne 0} \chi(s(x\odot y))\right|^2.$$
Note that the term of the RHS of this inequality is in turn equal to 
\begin{equation*}\label{Inequality_Omega}q^{-2} |E| \sum_{x\in M_2(\mathbb F_q)} \sum_{y, y'\in F, s,s'\ne 0}  \chi(x\odot (sy-s'y')).\end{equation*}
Using the orthogonality of $\chi$ for the Odot-product to compute the sum over $x\in M_2(\mathbb F_q)$, we obtain
$$|\Omega|^2\le q^2|E| \sum_{y,y'\in F, s,s'\ne 0: sy=s'y'} 1.$$
Considering the cases that $s=s'$ and $s\ne s'$, we have
$$|\Omega|^2\le q^2|E| \sum_{y\in F, s\ne 0} 1 + q^2|E| \sum_{\substack{y,y'\in F, s,s'\ne 0:\\s\ne s', sy=s'y'}} 1\le  q^3|E||F| +q^2|E| \sum_{\substack{y,y'\in F, s,s'\ne 0:\\s\ne s', sy=s'y'}} 1.$$
Whenever we fix $y\in F, s, s'\ne 0,$ there is at most one $y'\in F$ such that $sy=s'y'.$ Therefore, 
$$ |\Omega|^2 \le q^3|E||F| + q^4|E||F| \le 2 q^4|E||F|,$$
as desired.
\end{proof}

For two subsets $E,F$ of $M_2(\mathbb F_q)$,  we denote by $\Lambda(E,F)$ the additive energy defined by
$$\Lambda(E,F):=|\{(x,y,z,w)\in E\times F\times E\times F: x+y=z+w\}|.$$
The following proposition, whose proof will be given at the end of this section, plays a key role in the proof of Theorem \ref{mainthm2} below. 
\begin{proposition}\label{Proposition_Lambda}
 Assume that $E\subseteq D_i$ and $F\subseteq D_j$ for $i,j\ne 0.$ Then we have $$ \Lambda(E,F) \ll q^{-1}|E|^2|F| + q|E||F| + q |E|^{3/2}|F|^{1/2}.$$
\end{proposition}

The following is a main result of ours for the sum of two sets, whose proof heavily depends on Proposition \ref{Proposition_Lambda}
\begin{theorem}\label{mainthm2} Assume that $E\subseteq D_i$ and $F\subseteq D_j$ for $i,j\ne 0.$ Then we have
$$ |E+F|\gg 
\min\left\{ q|F|,~ \frac{|E||F|}{q},~ \frac{|E|^{1/2}|F|^{3/2}}{q}\right\}.$$
\end{theorem}
\begin{proof}
From the Cauchy-Schwarz inequality, it follows that
\begin{equation}\label{formula}|E+F|\ge \frac{|E|^2|F|^2}{\Lambda(E,F)},\end{equation}
By Proposition \ref{Proposition_Lambda}, we have 
$$ |E+F|\gg \frac{|E|^2|F|^2}{q^{-1}|E|^2|F|+q|E||F|+q|E|^{3/2}|F|^{1/2}}.$$
Then from this inequality, the proposition is immediate.

\end{proof}
\bigskip
In fact, in Theorem \ref{mainthm2}, if we know which one of $E$ and $F$ is larger than the other, then we can give a simpler statement.
\begin{corollary}\label{core} For $i,j\in \mathbb F_q^*,$ let $E\subseteq D_i$ and $F\subseteq D_j.$  Suppose, say,  $|F|\ge |E|$.
Then, we have 
$$ |E+F| \gg \min\left\{ q|F|,~ \frac{|E||F|}{q}\right\}.$$
\end{corollary}
\begin{proof}
Since $|F|\ge |E|$, we see that $|E||F|/q\le |E|^{1/2}|F|^{3/2}/q.$ Hence, the corollary follows immediately from Theorem \ref{mainthm2}.
\end{proof}
\subsection{Proof of  Proposition \ref{Proposition_Lambda}}
Here we give a proof of Proposition \ref{Proposition_Lambda}. We begin by giving a simple lemma.
\begin{lemma}\label{Cardinality bound}
Let $X$ be a finite set, and  $X=\bigcup_{k=1}^m X_k$ be a partition on $X$ with $a:=|X_k|=|X_\ell|$ for all $k,\ell.$  If $Y$ is a subset of $X$ such that $Y\cap X_k\ne \emptyset$ for any $k=1,...,m$, then the cardinality of $Y$ is bounded by  $$\frac{|X|}{a}\leq |Y|\leq \frac{b |X|}{a},$$ where $b:=\max_{1\leq k\leq m}\{|Y\cap X_k|\}.$
\end{lemma}
\begin{proof}
Notice that $m=\frac{|X|}{a}$ is the number of members of the partition.
Since $Y\cap X_k\ne \emptyset$ for any $k=1,...,m$, we have $ m\le |Y|.$ Since $|Y\cap X_k|\leq b$ for all $k=1,...,m$ and $|Y|=\sum_{1\leq k\leq m}|Y\cap X_k|$, we have $|Y|\le mb.$ This proves the lemma.
\end{proof}

Lemma \ref{Cardinality bound} is useful when we want to obtain a bound on the cardinality of a set $Y$ in question. It is enough to find a lager set $X$ allowing an embedding $Y\hookrightarrow X$ of sets satisfying the conditions in the lemma. Indeed, we will use this lemma at the last moment to complete the proof of Proposition \ref{Proposition_Lambda} below.  
\begin{proof}[Proof of Proposition \ref{Proposition_Lambda}]
Since $F\subseteq D_j,$  we can write \begin{equation} \label{Bound_Lambda}\Lambda(E,F)\le \sum_{x,z\in E, y\in F: \det(x+y-z)=j} 1=\sum_{x,z\in E, y\in F: (x+y)\odot (x-z)=0} 1.\end{equation}
Here the equality in \eqref{Bound_Lambda} follows from the equivalence of two conditions: for $x,z\in E\subseteq D_i, y\in F\subseteq D_j,$ 
 $$\det(x+y-z)=j \Leftrightarrow (x+y)\odot (x-z)=0.$$ 
To make the computation easy, we split the RHS of (\ref {Bound_Lambda})  into two summands:
$$\sum_{x,z\in E, y\in F: (x+y)\odot (x-z)=0} 1=I+II,$$
where $I$ denotes the sum over $x,y,z$ with  $\det(x+y)=0$ or $\det(x-z)=0$, 
and  $II$ the sum over $x,y,z$ with  $\det(x+y)\ne 0$ and $\det(x-z)\ne 0.$
Let us bound $I$ and $II$ separately.\\

For $I,$ the following is obvious.
$$ I \le \sum_{x,z\in E, y\in F: \det(x+y)=0} 1 + \sum_{x,z\in E, y\in F: \det(x-z)=0} 1$$ 
\begin{equation}\label{Bound_Lambda_3} =|E| \sum_{x\in E, y\in F: \det(x+y)=0} 1 + |F| \sum_{x,z\in E: \det(x-z)=0}1.\end{equation}
Lemma \ref{secondlem} directly gives a bound on the first sum in (\ref{Bound_Lambda_3}). To bound the second sum,  notice that since $E$ is a subset of $D_i$,  $-E$  is also contained in $D_i$ and $ |E|=|-E|.$ Thus Lemma \ref{secondlem} is also applicable to the second sum.
Therefore, we have obtained
$$ I\ll |E|( q^{-1}|E||F| + q|E|^{1/2} |F|^{1/2} ) + |F| ( q^{-1}|E|^2+q|E|)$$
$$\ll q^{-1}|E|^2|F| + q|E||F| + q|E|^{3/2}|F|^{1/2}.$$

Next, we bound  $II$. Recall that
$$ II=\sum_{\substack{x,z\in E, y\in F: (x+y)\odot (x-z)=0\\
                             \det(x+y)\ne 0 \ne \det(x-z)}} 1
=\sum_{x\in E} \left[\sum_{\substack{z\in E, y\in -F: (x-y)\odot (x-z)=0\\ \det(x-y)\ne 0 \ne \det(x-z)}} 1\right].$$
Fix $x\in E,$ and let $\beta=x-y$ and $\alpha=x-z.$ Then we see that
\begin{equation}\label{Bound_Lambda_4} II= \sum_{x\in E} \left[ \sum_{\substack{\alpha\in (-E+x), \beta\in (F+x):\\ \alpha\odot \beta=0, \det(\alpha)\ne 0 \ne\det(\beta)}} 1\right],\end{equation}
where $-E+x:=\{-e+x: e\in E\}$ and $F+x:=\{f+x: f\in F\}.$
Let  $II(x)$  be the sum in the bracket in (\ref{Bound_Lambda_4}); namely, 
$$ II(x):=\sum_{\substack{\alpha\in (-E+x), \beta\in (F+x):\\ \alpha\odot \beta=0, \det(\alpha)\ne 0 \ne\det(\beta)}}1.$$
Now for each $x\in E$ we bound  $II(x).$

For a nonzero vector  $\gamma \in M_2(\mathbb F_q)$,  let $[\gamma]$ be the one dimensional subspace ( i.e., the  line) in $M_2(\mathbb F_q)$ generated by $\gamma$ and $[\gamma]^*:=[\gamma]\setminus{\{\mathbf{0}\}}$.
For $H \subseteq  M_2(\mathbb F_q)$, let $${H}_x:=\{s\alpha: s\in \mathbb F_q^*, \alpha \in H+x, \det(\alpha)\ne 0\}.$$ In other words, ${H}_x$ is the union of all  $[\gamma]^*$ with $\gamma \in H+x, \det(\gamma)\ne 0. $
Notice that  for  $(\alpha, \beta)\in (-E+x)\times (F+x)$,   $\alpha \odot \beta =0$ iff  $(s\alpha)\odot (t\beta)=0$ for all $s,t\in \mathbb F_q^*$, and for any $\gamma\in \MFq$, $\det(\gamma)\ne0$ iff $\det (s\gamma)\ne 0$ for any nonzero $s\in \mathbb F_q.$
We claim that for every $x\in E,$ we have
\begin{equation}\label{claimMM} II(x)=\sum_{\substack{\alpha\in (-E+x), \beta\in (F+x):\\ \alpha\odot \beta=0, \det(\alpha)\ne 0 \ne\det(\beta)}}1 \sim q^{-2} \sum_{u\in (-E)_x, v\in {F}_x: u\odot v=0} 1.\end{equation}

To prove the claim, we use Lemma \ref{Cardinality bound}. 
Let $Y$ be the index set of the first summation in (\ref{claimMM}) which we want to count, and $X$ the index set of the second summation, i.e., 
$$X=\bigcup_{(\alpha,\beta)\in Y}[\alpha]^*\times[\beta]^*,$$
where we take the ordinary (not necessarily disjoint) union of sets. 
Obviously we have a natural embedding $Y\hookrightarrow X$, $(\alpha,\beta)\mapsto (\alpha,\beta).$ For the remaining conditions in Lemma \ref{Cardinality bound}, it is enough to show that  for any $(\alpha, \beta)\in Y,$ $$1\leq Y\cap ([\alpha]^*\times [\beta]^*)\leq 4.$$
The inequality $1\leq Y\cap ([\alpha]^*\times [\beta]^*)$ is trivially true. For the other inequality, it is enough to show two inequalities
 \begin{equation}\label{claimm} |[\alpha]^*\cap (-E+x)|\le 2, \ \ \ |[\beta]^*\cap (F+x)|\le 2.\end{equation} 
We only prove the first one in \eqref{claimm} (in fact, the proof below works for the second inequality.)
First, notice that $-E\subseteq D_i$ since $E\subseteq D_i,$ and so   $-E+x\subseteq D_i+x.$
Since $|[\alpha]^*\cap (-E+x)|\le |[\alpha]^*\cap (D_i+x)|$, it suffices to show 
$$|[\alpha]^*\cap (D_i+x)|\le 2.$$

Note that for a  (fixed)  $ x=\left(\begin{matrix}
x_1&x_2\\
x_3&x_4
\end{matrix}\right)\in E,$ the variety $D_i+x$ is defined by the equation $$(z_1-x_1)(z_4-x_4)-(z_2-x_2)(z_3-x_3)=i,$$ where $z=\left(\begin{matrix}
z_1&z_2\\
z_3&z_4
\end{matrix}\right) \in M_2(\mathbb F_q).$

Therefore, for  $\alpha=\left(\begin{matrix}
\alpha_1&\alpha_2\\
\alpha_3&\alpha_4
\end{matrix}\right)$,  an element $s\alpha \in [\alpha]^* $ for $s\in \mathbb F_q^*$ lies in the variety $[\alpha]^*\cap (D_i+x)$ if and only if $s$ satisfies
$$ (s\alpha_1-x_1)(s\alpha_4-x_4)-(s\alpha_2-x_2)(s\alpha_3-x_3)=i;$$ equivalently,  
\begin{equation}\label{quadratic equation} (\alpha_1\alpha_4-\alpha_2\alpha_3)s^2+(\alpha_2x_3+\alpha_3x_2-\alpha_1x_4-\alpha_4x_1)s +x_1x_4-x_2x_3-i=0.\end{equation} In other words,
 the number $|[\alpha]^*\cap (D_i+x)|$ is  equal to the number of solutions to this  equation (\ref{quadratic equation}) for $s$. Since $\det(\alpha)\ne 0,$ this equation is quadratic, and so it has at most two solutions. Thus we have $$|[\alpha]^*\cap (D_i+x)|\leq 2,$$ as desired.  Hence, the inequality \eqref{claimm} holds. Note that the number of the above partitions on $X$ is equal to $\frac{X}{(q-1)^2}$. From Lemma \ref{Cardinality bound}, it follows that $$\frac{|X|}{(q-1)^2}\leq II(x)\leq 4\frac{|X|}{(q-1)^2}.$$ This proves the claim \eqref{claimMM}.
\smallskip

Now we are ready to bound $II(x)$ in \eqref{claimMM}. It is clear that 
$$ |(-E)_x|\sim q|(-E+x)|,~|{F}_x|\sim q|(F+x)| \quad \mbox{for all}~x\in E.$$  
Applying the inequality \eqref{zeroK} in Lemma \ref{odotzero}, we see that for every $x\in E,$
$$II(x) \ll q^{-2} \left( \frac{|(-E)_x||{F}_x|}{q} + q^2 |(-E)_x|^{1/2}|{F}_x|^{1/2}\right)$$
$$\sim \frac{|E||F|}{q} + q |E|^{1/2}|F|^{1/2}.$$
Summing over $x\in E,$ 
$$ II\ll \frac{|E|^2|F|}{q} + q |E|^{3/2}|F|^{1/2}.$$
To conclude,
$$ \Lambda(E,F) \le I +II \ll (q^{-1}|E|^2|F| + q|E||F| + q|E|^{3/2}|F|^{1/2})+(q^{-1}|E|^2|F| + q |E|^{3/2}|F|^{1/2})$$
$$\ll q^{-1}|E|^2|F| + q|E||F| + q |E|^{3/2}|F|^{1/2},$$ as desired.
\end{proof}

\bigskip
\section{Determinants of finitely iterated sum sets (Proof of Theorem \ref{thm0})}\label{sec6N}
As we will see, the proof of Theorem \ref{thm0} uses some other results  as well as Corollary \ref{core}. We will list them below.
The following result was given by Li and Su \cite{li} by using Fourier techniques. A graph theoretic proof was recently given by Demirogly Karabulut \cite{Yesim}. For the sake of completeness, we will include a short proof in Appendix. 

\begin{proposition}[\cite{li, Yesim}]\label{thm3} Let $E, F \subseteq M_2(\mathbb F_q).$ If $|E||F|> 4 q^5,$ then 
we have 
$$\det(E+F) \supseteq \mathbb F_q^*.$$
\end{proposition}
The following result is an immediate consequence from Corollary \ref{core} for the balance case.
\begin{lemma}\label{corecol} For $i\in \mathbb F_q^*,$ let $E\subseteq D_i.$ Then we have
$$ |E+E| \gg \min\left\{ q|E|,~ \frac{|E|^2}{q}\right\}.$$
\end{lemma}
We note that Corollary \ref{core} only gives us the lower bound when $F$ is a set in $D_j$ for some $j\in \mathbb{F}_q^*$. To make the inductive argument in the proof of Theorem \ref{thm0} below work, we also need the following result from \cite{Ye} in the case when $F$ is an arbitrary set in $M_2(\mathbb{F}_q)$. We refer readers to \cite{Ye} for a detailed proof using spectrum of the unit-special Cayley graph.

\begin{lemma}[Proof of Corollary $1.7$, \cite{Ye}]\label{thm090}
For $i\in \mathbb{F}_q^*,$ let $E$ be a set in $D_i$, and $F$ be a set in $M_2(\mathbb{F}_q).$ Then we have 
\[|E+F|\gg \min \left\lbrace q|E|, ~\frac{|E|^2|F|}{q^3}\right\rbrace.\]
\end{lemma}

It is worth noting that the bound in Corollary \ref{core} is stronger than that of Lemma \ref{thm090} whenever $|E|\le q^2$. Another key ingredient in proving Theorem \ref{thm0} is the following lemma whose proof is based on
an induction argument with  Lemma \ref{corecol} and Lemma \ref{thm090}.
\begin{lemma}\label{core1}
Let $k\ge 2$ be an integer and $i$ be an element in $\mathbb{F}_q^*$.  Let $E$ be a set in $D_i$ with $|E|\ge C q^{\frac{3}{2}}$ for a sufficiently large constant $C.$  We have \[|kE|\gg \min\left\lbrace  q|E|, ~\frac{|E|^{2k-2}}{q^{3k-5}}\right\rbrace.\]
\end{lemma}
\begin{proof}
The proof proceeds by induction on $k$. 
Suppose $k=2.$ Then Lemma \ref{corecol} gives us 
\[|E+E|\gg \min \left\lbrace q|E|,~ \frac{|E|^2}{q}\right\rbrace.\]
Thus the base case follows.
Suppose that the theorem holds for any $k-1\ge 2$. We now show that it also holds for $k$. Indeed, by inductive hypothesis, we have 
\[|(k-1)E|\gg \min \left\lbrace q|E|, ~\frac{|E|^{2k-4}}{q^{3k-8}}\right\rbrace.\]

Applying Lemma \ref{thm090} with $F=(k-1)E$, we have 
\[|kE|\gg \min \left\lbrace q|E|, ~\frac{|E|^3}{q^2}, ~ \frac{|E|^{2k-2}}{q^{3k-5}}\right\rbrace\gg  \min \left\lbrace q|E|, ~ \frac{|E|^{2k-2}}{q^{3k-5}}\right\rbrace,\]
since $|E|\ge C q^{\frac{3}{2}}$. This concludes the proof of Lemma \ref{core1}. 
\end{proof}
\paragraph{\textbf{Proof of Theorem \ref{thm0}}:}
From Lemma \ref{core1}, we see that one of the following cases happens. 

{\bf Case $1$:} If $|kE|\gg q|E|$, then by applying Proposition \ref{thm3} for the set $kE,$  we have 
\[\det(2kE)\supseteq \mathbb{F}_q^*,\]
whenever $|E|\ge C q^{\frac{3}{2}}$. 

{\bf Case $2$:} If $|kE|\gg \frac{|E|^{2k-2}}{q^{3k-5}}$, then we apply Proposition \ref{thm3} again to obtain
\[\det(2kE)\supseteq \mathbb{F}_q^*,\]
whenever $|E|\ge C q^{\frac{6k-5}{4k-4}}$.

This completes the proof of the theorem. $\hfill\square$

\section{Appendix}
In this appendix, we gives an alternative proof of  Proposition \ref{thm3}.
We begin by proving a preliminary lemma below.
 \begin{lemma} For $t\in \mathbb F_q, m\in  M_2(\mathbb F_q),$ we have
 \begin{equation}\label{eq2} \widehat{D_t}(m)=\frac{\delta_0(m)}{q} + \frac{1}{q^3} \sum_{s\ne 0} \chi\left( -st-\frac{\|m\|_*}{s}\right),\end{equation}
 where $\delta_0(m)=1$ if $m=\mathbf{0},$ and 0 otherwise.
  \end{lemma}
 \begin{proof} By the orthogonality of $\chi$, we have
 $$ \widehat{D_t}(m)=q^{-4}\sum_{x\in M_2(\mathbb F_q):\|x\|_*=t} \chi(-m\cdot x)=q^{-5} \sum_{x\in M_2(\mathbb F_q)} \sum_{s\in \mathbb F_q} \chi(s(\|x\|_*-t)) \chi(-m\cdot x)$$
$$= \frac{\delta_0(m)}{q} + q^{-5}\sum_{s\ne 0} \sum_{x\in M_2(\mathbb F_q)} 
\chi(s(\|x\|_*-t)) \chi(-m\cdot x)$$
$$=\frac{\delta_0(m)}{q} +  q^{-5}\sum_{s\ne 0} \chi(-st) \left(\sum_{x_1,x_4\in \mathbb F_q} \chi( (sx_4-m_1)x_1 -m_4x_4)\right) \left(\sum_{x_2,x_3\in \mathbb F_q} \chi((-sx_3-m_2)x_2-m_3x_3)\right)$$
Using the orthogonality of $\chi$ again, we compute the sums over $x_1, x_3\in \mathbb F_q.$ Then we see that 
$$\widehat{D_t}(m)=\frac{\delta_0(m)}{q} + q^{-3}\sum_{s\ne 0} \chi(-st) \chi\left(-\frac{\|m\|_*}{s}\right),$$ 
which completes the proof.
 \end{proof}
 \begin{proof}[Proof of Proposition \ref{thm3}]
To complete the proof, it will be enough to show that if $|E||F|>4 q^5,$ then $N_t(E,F)>0$ for all $t\in \mathbb F_q^*.$ We proceed as in \cite{IR07}.
By definition, 
$$N_t(E,F)=\sum_{x\in E, y\in F: \det(x+y)=t} 1 = \sum_{x\in E, y\in F} D_t(x+y). $$
Applying the Fourier inversion theorem to the function $D_t(x+y)$ and using the definition of the Fourier transform, we see that
\begin{equation}\label{eq1} N_t(E,F)=q^{8} \sum_{m\in  M_2(\mathbb F_q)} \widehat{D_t}(m) \overline{\widehat{E}}(m) \overline{\widehat{F}}(m).\end{equation}
 Combining \eqref{eq1} with \eqref{eq2}, we get
 $$ N_t(E,F)= \frac{|E||F|}{q} + R(t),$$
 where $R(t)$ is given by
 $$ R(t)=q^5 \sum_{m\in M_2(\mathbb F_q)} \left(\sum_{s\ne 0} \chi\left( -st-\frac{\|m\|_*}{s}\right)\right) \overline{\widehat{E}}(m) \overline{\widehat{F}}(m).$$
For $t\ne 0$, the sum over $s\ne 0$ is the Kloosterman sum whose absolute value is less than or equal to $2\sqrt{q}.$  Thus we have
$$ N_t(E,F)\ge \frac{|E||F|}{q} -|R(t)| \ge\frac{|E||F|}{q}- 2q^{11/2} \sum_{m\in \mathbb F_q^4} |\widehat{E}(m)| |\widehat{F}(m)|$$
$$ \ge \frac{|E||F|}{q} - 2 q^{3/2} |E|^{1/2}|F|^{1/2},$$
where the last inequality follows from the Cauchy-Schwarz inequality and the Plancherel theorem. Thus $N_t(E,F)>0,$ provided that $|E||F|>4 q^5.$  This completes the proof. 
 \end{proof}

\end{document}